\documentclass[11pt,leqno]{article}
\usepackage{tikz}

\usepackage{amssymb,amsmath,amsthm,mysects,url,rotating}
 \usepackage{graphicx,epsfig}
 \usepackage{xcolor,graphicx}
\newcommand{\n}{\noindent}

\newcommand{\vp}{\varepsilon}
\newcommand{\bb}[1]{\mathbb{#1}}
\newcommand{\cl}[1]{\mathcal{#1}}

\newcommand{\ovl}{\overline}

\theoremstyle{plain}
\newtheorem{thm}{Theorem}[section]
\newtheorem{lem}[thm]{Lemma}

\newtheorem{pro}[thm]{Proposition}

\newtheorem{cor}[thm]{Corollary}

\theoremstyle{definition}

\newtheorem{dfn}[thm]{Definition}

\theoremstyle{remark}
\newtheorem{rem}[thm]{Remark}

\numberwithin{equation}{section}

\def\tilde{\widetilde}

\renewcommand{\tilde}{\widetilde}

\def\RR{\bb R}
\def\R{\bb R}
\def\Z{\bb Z}

\def\C{\bb C}

\def\E{\bb E}
\def\EE{\bb E}
\def\N{\bb N}
\def\P{\bb P}

\def\T{\bb T}

\def\hat{\widehat}

\begin{document}
\def\d{\delta}
\def\n{\nolimits}

 \title{Subgaussian sequences   in probability and Fourier analysis}

\author{by\\
Gilles  Pisier\\
Texas A\&M University and UPMC-Paris VI}

\maketitle
\begin{abstract} This is a review on subgaussian sequences
of random variables, prepared for the 
Mediterranean Institute for the Mathematical Sciences (MIMS).
 We first describe the main examples
of such sequences. Then we focus on examples coming
from the harmonic analysis of Fourier series  and we describe the connection
of  subgaussian sequences of characters on the unidimensional torus  
(or any compact Abelian group) with Sidon sets.
We explain the main combinatorial open problem concerning
such subgaussian sequences. 
We present the answer to the  analogous question 
for subgaussian bounded mean oscillation (BMO) sequences on the unit circle. Lastly, we describe
several very recent
  results that provide a generalization of the preceding ones
 when the trigonometric system (or its analogue on a compact Abelian group)
 is replaced by an arbitrary orthonormal system
 bounded in $L_\infty$.
 \end{abstract}
 
 \def\tr{{\rm tr}}

\tableofcontents
\newpage

A  sequence $(f_n)$ of  real valued random variables is called subgaussian  
 if there is $s\ge 0$ such that 
 for any finitely supported $(x_n)\in \R^\N$ 
  \begin{equation}\label{40bb} \E \exp{(\sum x_n f_n)}\le \exp ({s^2 \sum x_n^2/2}).\end{equation}

The equality case corresponds to Gaussian independent variables
with the same variance.
A similar definition (see below) can be given for the $\C$-valued case.
Then the family is subgaussian if and only if (iff in short)
the family that is the union of the real and imaginary
parts of $(f_n)$ is subgaussian in the real sense.

As we will show, this notion plays an important role in Gaussian process theory
and in the harmonic analysis of thin sets, such as Sidon sets.
In fact, as will be shown in \S \ref{sid}, a  subsequence
of the  trigonometric system of the form $f_n(t)=\exp{(ik(n)t)}$
(with $k(n)$ distinct integers in $\Z$)
is subgaussian on $ ([0,2\pi], dt/2\pi)$  iff it is a Sidon sequence,
i.e. one for which  any continuous function $\varphi$ on
the unit circle (identified as usual with  $\R/2\pi\Z$)
with Fourier transform $\hat \varphi$ supported by the set $\{k(n)\}$
has an absolutely convergent Fourier series
$$\varphi(t)=\sum \hat \varphi (k(n)) \exp{(ik(n)t)}.$$
 It turns out that much of
the connection between subgaussian and Sidon sequences
remains valid for general uniformly bounded orthonormal
systems. This came as a  surprise since it was generally believed
that the group structure played a key role. This very recent development
from \cite{BoLe,Pi}
is described in \S \ref{bole}.

 The important feature
of subgaussian sequences is that 
although they share many properties  of bounded independent random variables,
they actually seem much   more general. 
The notion of subgaussian   seems somewhat transversal 
in probability theory : it interacts with many fundamental topics
such as Gaussian processes, martingales, Orlicz spaces, Fourier series or isoperimetric inequalities  (see e.g. \cite{MaPi,Pip,Ka2,LeTa,LQ,Pi4}) 
but it can never be reduced to the intersection with any of these topics.
As we will explain in \S \ref{mop}, there is 
a major open problem that proposes a characterization
of subgaussian sequences in the Fourier series framework.
The true meaning of subgaussian remains puzzling.
The more recent
results on
uniformly bounded orthonormal
systems   described at the end of the paper
give some hope to make progress to clarify that.

\section{Gaussian and subgaussian variables}
In this paper, a real valued Gaussian   random variable $g$
on a probability space $(\Omega, \P)$ is called
Gaussian if there is $\sigma \ge 0$ such that
for any measurable $A\subset \R$
$$\P\{g\in A\}=\int_A e^{-x^2/2\sigma^2}dx/\sqrt{2\pi}\sigma.$$
Note that \emph{we only consider Gaussian variables with mean $0$}.\\
Then  $\sigma^2$ is   the variance of $g$ and $\E g^2=\sigma^2$.
When $\sigma=1$, $g$ is called normal.
We have then
$$\forall z\in \C\quad \E \exp(z g)=\exp z^2/2.$$

A complex valued  random variable $g$ is called
$\C$-Gaussian (resp. $\C$-Gaussian normal) if its real and imaginary parts
are independent Gaussian with the same variance $\sigma$
(resp. with variance $1$).  
We have then when $\sigma=1$
$$\forall z\in \C\quad \E \exp(\Re(z g))=\E \exp(\Re(\bar z g))=\exp |z|^2/2.$$
Warning: with this convention,  a nonzero real valued  Gaussian variable is not $\C$-Gaussian !

We also need a variant: a $\C$-Gaussian  variable will be called
normalized if $\E|g|^2=1$ (note that for a 
normal $\C$-Gaussian variable we have $\E|g|^2=2$).

For convenience, we  will sometimes call $\R$-Gaussian
any real valued Gaussian random variable. We will say that it is normalized
if its $L_2$-norm is $1$. In the real case this is the same as   normal.

Let $(g_n)$ be an i.i.d. sequence of  
 normalized $\R$-Gaussian (resp. $\C$-Gaussian)  variables. Note that this is an orthonormal
system in $L_2(\Omega, \P)$.
Then for any (nonzero) sequence $x=(x_n)\in \ell_2$, the variable
$g=(\sum |x_n|^2)^{-1/2} \sum x_n g_n$ is 
a standard Gaussian variable. Therefore
 \begin{equation}\label{74}\|\sum x_n g_n\|_p= \|g_1\|_p  (\sum |x_n|^2)^{1/2}.\end{equation}
In the real case (with $x_n\in \R$ $\forall n$)  
 \begin{equation}\label{75}\E \exp (\sum x_n g_n)= \exp( \sum x_n^2/2).\end{equation}
In the complex case, assuming  $(g_n)$ $\C$-Gaussian normal (with $x_n\in \C$ $\forall n$) 
 \begin{equation}\label{75b}\E \exp (\Re(\sum x_n g_n))= \exp( \sum |x_n|^2/2).\end{equation}
\begin{dfn} A real valued random variable $f$
is  called subgaussian if 
there is a  constant $s\ge 0$ such that
for any $x\in \R$
 \begin{equation}\label{40a} \E \exp{xf}\le \exp {s^2 x^2/2}.\end{equation}
 
 As is well known this implies
 that for any $c>0$ 
  \begin{equation}\label{102}\P(\{f>c\})\le \exp{-(c^2/2s^2)}\end{equation}
 and also
 \begin{equation}\label{103}\P(\{f<-c\})\le \exp{-(c^2/2s^2)}.\end{equation}
 Indeed, by Markov's inequality we have for any $x>0$
 $\P(\{f>c\})\le  \exp {(s^2 x^2/2-xc)}$
 and the choice of $x=c/s^2$ yields \eqref{102}. Then
 \eqref{103} follows by applying \eqref{102} to $-f$.

  A complex  valued random variable $f$
is  called $\C$-subgaussian if 
there is  constant $s\ge 0$ such that
for any $x\in \C$
 \begin{equation}\label{40aa} \E \exp{\Re(xf)}\le \exp {s^2 |x|^2/2}.\end{equation}
 A  real valued sequence $(f_n)$ is called subgaussian if
 if there is $s\ge 0$ such that 
 for any $(x_n)\in \R^\N$ in the unit sphere of $\ell_2$
 the variable $f=\sum x_n f_n$ satisfies \eqref{40a}. Equivalently,
 for any finitely supported $(x_n)\in \R^\N$ 
  \begin{equation}\label{40b} \E \exp{(\sum x_n f_n)}\le \exp ({s^2 \sum x_n^2/2}).\end{equation}
 A  complex valued sequence $(f_n)$ is called $\C$-subgaussian if
 the real valued sequence formed together  by both its real parts $(\Re f_n)$
 and its imaginary parts $(\Im f_n)$ is subgaussian in the preceding
 sense. This implies that for some $s\ge 0$
  for any finitely supported $(x_n)\in \C^\N$ 
  \begin{equation}\label{40b'} \E \exp{(\Re(\sum x_n f_n))}\le \exp ({s^2 \sum |x_n|^2/2}).\end{equation}
Moreover, we denote by  
$sg(f)$ (resp. $sg(\{f_n\})$)  the   smallest number
$s\ge 0$ for which this holds.
\end{dfn}

The following are immediate consequences of the definition:
\begin{lem}\label{im} If $f$ is $\R$-subgaussian (resp. $\C$-subgaussian)
then so is $tf$ for any $t\in \R$ (resp. $t\in \C$)
and $sg(tf)=|t| sg(f)$. Also
$\E(f)=0$ and in the real case $\E f^2\le sg(f)^2$. 
\\
Let $f_1,f_2$ be two subgaussian variables (either both  real or both complex).
Then
\begin{equation}\label{69} sg(f_1+f_2)\le \sqrt {2} (sg(f_1)^2+sg(f_2)^2  )^{1/2}.\end{equation}
Moreover, if $(f_n)$ is
an independent sequence of $\R$-subgaussian (resp. $\C$-subgaussian)
variables such that $\sum sg(\{f_n\})^2 <\infty$,
then $f=\sum f_n$ is  $\R$-subgaussian (resp. $\C$-subgaussian)
with $sg(f)\le (\sum sg(\{f_n\})^2)^{1/2} $.
\end{lem}
\begin{proof}
\eqref{69} follows from the easy (and soft) observation
 that if in the real valued case
   $s_1=sg(f_1)$ and $s_2=sg(f_2)$, 
 we have     by Cauchy-Schwarz for any $x\in \R$
 $$\int \exp(x (f_1+f_2)) dm\le \left(\int \exp(2x f_1) dm\int \exp(2x f_2) dm\right)^{1/2}
 \le  \left( \exp(2x^2s_1^2 )   \exp(2x^2s_2^2 ) \right)^{1/2} $$
$$=\exp(x^2(s_1^2+s_2^2 )   ).$$
The other assertions are left to the reader.
\end{proof}

Concerning \eqref{69}, we will show later (see Lemma \ref{42}) that $f\mapsto sg(f)$
is equivalent to a norm, namely $f\mapsto\|f\|_{\psi_2}$.

In the real valued case we sometimes use the term $\R$-subgaussian
instead of subgaussian.

Of course,  $\R$-Gaussian (resp. $\C$-Gaussian) implies
$\R$-subgaussian (resp. $\C$-subgaussian), and
for a normal Gaussian variable $g$ we have $sg(g)=1$.

A simple and basic non-Gaussian example is given by a sequence $(\vp_n)$
of independent choices of signs $\vp_n=\pm 1$  taking the values
$\pm 1$
with equal probability  $1/2$. Then one has $sg(\{\vp_n\})= 1$.
This follows simply from  
\begin{equation}\label{chp6aeqH22}
\forall x\in \R\quad
\cosh({x}) \le \exp({x}^2/2),
\end{equation}
which just follows from Stirling's formula:
\[
\cosh({x}) = 1+\sum\nolimits^\infty_1 {x}^{2n}/(2n)! \le 1 + 
\sum\nolimits^\infty_1 {x}^{2n}/(2^nn!).
\]
More generally, by an inequality due to Azuma \cite{Az},
martingale increments satisfy the same:

\begin{thm}\label{chp6athmH5} Let $(f_n)_{n\ge 0}$
be a  real valued martingale  in $L_1$ on some probability space.
Let $d_n=f_n-f_{n-1}$ ($n\ge 1$).
Then if $\|d_n\|_\infty\le 1$ for any $n\ge 1$,
the sequence $(d_n)$ is subgaussian with $sg(\{d_n\})\le 1$.
\end{thm}
\begin{proof}
We will use the following elementary bound:\ for any $t\in \R$
\begin{equation}\label{chp6aeqH21}
\forall d\in [-1,1]\qquad\qquad \exp({x} d) \le \cosh({x}) + 
d\sinh({x}).
\end{equation}
Indeed, by the convexity of $d\to \exp({x} d)$ on $[-1,1]$, since $d = 
2^{-1}(d+1)(1) + 2^{-1}(1-d)(-1)$ we have 
\[
\exp({x} d) \le 2^{-1}(d+1) \exp({x}) + 2^{-1}(1-d) \exp(-{x}),
\]
which proves this bound.

Let $M_n=\sum\n_1^n x_k d_k$.
Clearly $(M_n)$ is a martingale relative to the filtration
associated to $(f_n)$. We denote by ${\bb E}_{n} $
the conditional expectation with respect to
$\sigma\{M_k\mid k\le n\}$ and we set $M_0=0$.
We now claim that for any $n\ge 1$
\[
{\bb E}_{n-1} \exp(M_n) \le \exp(M_{n-1}) \exp\left(  
x_n^2 /2\right).
\]
Note $M_n-M_{n-1}=t_nd_n$. We have 
by \eqref{chp6aeqH21} and  by \eqref{chp6aeqH22}
\begin{align*}
{\bb E}_{n-1} \exp M_n 
&\le \exp(M_{n-1}) {\bb E}_{n-1}[\cosh(x_n) + d_n \sinh(x_n)]\\
&= \exp(M_{n-1}) \cosh({x}_n)\\
&\le \exp(M_{n-1}) \exp(x^2_n/2)
\end{align*}
which proves the claim. Now

\[
{\bb E} \exp(M_n) = {\bb E}{\bb E}_{n-1} \exp M_n \le {\bb E} \exp (M_{n-1}) 
\exp(x^2_n/2),
\]
and hence by induction
\[
{\bb E} \exp(M_n) \le   \exp\left(\sum\nolimits^n_1 
x_k^2 /2\right).
\]
 \end{proof}
\begin{rem}\label{rr1}
The most basic example of subgaussian sequence
is a sequence $(\vp_n)$ of independent choices of signs,
i.e. an i.i.d. sequence of $\pm 1$-valued variables
with $\P(\{\vp_n=\pm1\})=1/2$. This classical
example is of course included in those given by the preceding statement
since the partial sums $S_n=\sum\n_1^n \vp_k$ form a martingale.
Note that 
\begin{equation}\label{vp} sg(\{\vp_n\})=1\quad {\rm and}\quad sg(\sum\n_1^n \vp_k)\le \sqrt n.\end{equation}
The complex analogue of $(\vp_n)$ is a sequence $(z_n)$
of i.i.d. random variables with values in
the unit circle $\T $ of $\C$ with distribution equal
to the normalized Haar measure on $\T$.
This sequence is $\C$-subgaussian with
$sg(\{z_n\})\le 1$.
Indeed, for any finitely supported $(x_n)\in \C^\N$,
the variables $(d_n)$ defined by
$d_n=\Re(x_n z_n) |x_n|^{-1} $
(with the convention $0/0=0$), being independent 
with mean $0$ form a sequence of martingale differences
with $|d_n|\le 1$. Thus by Theorem \ref{chp6athmH5}
$sg(\{d_n\})\le 1$, which implies 
$sg( |x_n|d_n)\le |x_n|$. 
Now by Lemma \ref{im},  if $\sum |x_n|^2=1$ then
$sg(\sum x_n z_n)\le 1$. Thus
we conclude that $sg(\{z_n\})\le 1$.
\end{rem}
Another important example of subgaussian random variable
can be derived from the fundamental isoperimetric inequality
for Gaussian measure and the related concentration phenomenon:
\begin{thm}\label{mau} Let $F:\ \RR^n \to \R^n$ be a mapping 
(a priori non-linear) satisfying the Lipschitz condition:
 \begin{equation}\label{95}\forall x,y\in \R^n\quad \|F(x)-F(y)\|_2\le \|x-y\|_2.\end{equation}
Let $(g_1,\cdots,g_n)$ be i.i.d. normal $\R$-Gaussian variables.
Then the variables
$$f_j=F_j(g_1,\cdots,g_n)- \EE F_j(g_1,\cdots,g_n)$$
are subgaussian with $sg(\{f_j\})\le 1$.
\end{thm} We will give two proofs. First following
   \cite[p. 181]{Pip} we review a proof due to Maurey using Brownian stochastic
 integrals and Azuma's inequality \eqref{chp6athmH5}.
 A similar proof already appears in \cite[p. 26]{CIS} (but we
 were not aware of that reference at the time \cite[p. 181]{Pip} was written).
  See also \cite{CIS, SuT},  for closely related results. See also
  the exposition in  \cite[chap. 3]{MI}, for the connection
  with isoperimetric inequalities.
  
  Let us sketch Maurey's argument.
 Fix $x\in \R^n$ with $\|x\|_2=1$.
 It suffices to show that the variable $\Phi=\sum\n_{j=1}^n x_j F_j(g_1,\cdots,g_n)$
 is subgaussian with $sg(\Phi)\le 1$.
 This 
rests on the formula
\begin{equation}\label{96}\Phi(B_1)-\E \Phi (B_1)=
\int_0^1 \nabla(P_{1-t}\Phi)(B_t) . dB_t, \end{equation}
where $(B_t)$ is the standard Brownian motion
 starting at $0$ on $\R^n$, and  $P_t$ is the associated Markov semigroup.
 By Lebesgue's classical differentiation results, we know
 that \eqref{95} implies $\|\nabla(\Phi)\|_2\le 1$ a.s., but
 since $  P_{1-t}F $   still satisfies \eqref{95},
 we also have $\|\nabla(P_{1-t}\Phi)\|_2\le 1$ a.s. and 
 we can  rewrite \eqref{96} as
 \begin{equation}\label{97} \Phi(B_1)-\E \Phi (B_1)=\int_0^1 V_t . dB_t\end{equation}
 with $(V_t)$ such that $\|V_t\|_2\le 1$ a.s. for all $0<t<1$.
 Fix $x\in \R$.
  Now 
  easy arguments from stochastic integration tell us that  the process
 $ M_s=\exp{( x\int_0^s V_t . dB_t-x^2 s/2)} $ ($0\le s\le 1$) is a supermartingale  
and hence
$$\E M_1\le \E M_0=1.$$
This last inequality means that $sg(\Phi)\le 1$, which proves Theorem \ref{mau}.

The second proof (also from \cite{Pip}) is very simple and more elementary 
but it only shows that $sg(\{f_j\})\le (\pi/2)^2$.
 It runs as follows.
Let $g'=(g'_1,\cdots,g'_n)$ be an independent copy of $g=(g_1,\cdots,g_n)$.
Then, let $g(t)=g\sin(t) +g' \cos(t)$. Note $g(\pi/2)=g$
and $g(0)=g'$. Let $g'(t)= \frac{d}{dt} g(t)=g\cos(t)-g'\sin(t)$.
The key observation is that for any $t$ the  pair $(g,g')$ has the 
 same distribution as $(g(t),g'(t))$ (indeed these are Gaussian random vectors in $\R^{2n}$
 with the same covariance). Then the proof boils down to ``the fundamental
 formula of calculus", namely
 $$\Phi (g)-\Phi (g')=\Phi (g(\pi/2))-\Phi (g(0))=\int_0^{\pi/2}  \frac{d}{dt} \Phi ( g(t)  )dt =
 \int_0^{\pi/2}  \nabla \Phi ( g(t)  ). g'(t) dt . $$
 Then by the convexity of the exponential function 
  \begin{equation}\label{m73}\E\exp{(\Phi (g)-\Phi (g'))} \le (2/\pi)  \int_0^{\pi/2}  \left(\E\exp{(\frac{\pi}{2}  \nabla \Phi ( g(t)  ). g'(t) )} \right)    dt,\end{equation}
 but by the distributional invariance of $(g(t),g'(t))$, 
  we have by \eqref{75} 
 $$\forall t\quad \E\exp{(\frac{\pi}{2}  \nabla \Phi ( g(t)  ). g'(t) )}=\E\exp{(\frac{\pi}{2}  \nabla \Phi ( g   ). g'  )}=\E \exp {( (\frac{\pi}{2})^2 \|\nabla \Phi ( g   )\|_2^2/2 )}\le \exp {( (\frac{\pi}{2})^2 /2 )},$$
 and hence by \eqref{m73}
$$\E\exp{(\Phi (g)-\Phi (g'))} \le \exp {( (\frac{\pi}{2})^2 /2 )}.$$
This means that
$sg(\{F_j(g) -F_j(g')\} ) \le (\pi/2)^2$. 
Since, again by convexity of the exponential, we have
 $\E\exp{(\Phi (g)-\E\Phi (g))} \le  \E\exp{(\Phi (g)-\Phi (g'))}   $,
 we obtain a fortiori $sg(\{F_j(g) -\E F_j(g)\} ) \le (\pi/2)^2$.
 
 \section{The Mehler kernel (Ornstein-Uhlenbeck semigroup)}\label{meh}
 
For further use at the end of this paper, we need to describe some basic facts
about the Mehler kernel. 
Let $\{g_n\mid 1\le n\le N\}$ be an i.i.d. sequence of  
 normalized $\R$-Gaussian   variables on 
 $(\Omega, {\cl A}, \P)$, where $\cl A$ is the $\sigma$-algebra
 generated by $\{g_n\mid 1\le n\le N\}$.\\
 Let $(h_n)$ ($n\ge 0$) be the Hermite polynomials on $\R$. 
 Recall $h_0=1$, $h_1(x)=x$. For any $\alpha=(n(1),\cdots,n(N))\in \N^N$, let $h_\alpha(x_1,\cdots,x_N)=h_{n(1)}(x_1)\cdots h_{n(N)}(x_N)$. We call $d=n(1)+\cdots+n(N)$ the degree of $h_\alpha$.
 It is well known that the family of Hermite polynomials
 (suitably normalized) $\{h_\alpha(g_1,\cdots,g_N)\}$
 forms an orthonormal basis of $L_2(\P)$.
 Let $P_0$ be the orthogonal projection onto
 the constant functions, and let $P_1$
 be the orthogonal projection onto {span}$[g_n\mid 1\le n\le N]$.
 More generally, we denote by 
 $P_d$   the orthogonal projection onto the span
 of the Hermite polynomials of degree $d$ in $\{g_n\mid 1\le n\le N\}$.
For any $\d\in [-1,1]$ the  operator  
$T_\d:\ L_2(\P) \to L_2(\P)$ defined
by
$$T_\d=\sum\n_0^\infty \d^d P_d$$
is a positive contraction on $L_p(\P)$
for all $1\le p\le \infty$.

It is well known that for any smooth enough (e.g. polynomial) function
$F(g_1,\cdots,g_N)$ in $L_1(\P)$ we have
$$(T_\d F) (g) = \E_{g'} F(\d g+ (1-\d^2)^{1/2} g')$$
where $g'=(g'_n)$ is an independent copy of  $\{g_n\mid 1\le n\le N\}$.
 This is sometimes called Mehler's formula.
 The operators $t\mapsto T_{e^{-t}}$ form the famous Ornstein-Uhlenbeck semigroup.
 
 It is an easy exercise to show
 that  if $-1<\d<1$ the operator $T_\d $ is given by a positive kernel
 $K_\d\in L_1(\P\times \P)$,
 in the sense that for any
 polynomials $F_1,F_2$ we have
 $$\langle T_\d(F_1) , F_2\rangle=\E_g \E_{g'}  K_\d(g,g')F_1(g')F_2(g) .$$
 
 Note that 
 $$\|K_\d\|_{L_1(\P\times \P)}=\langle T_\d(1) , 1\rangle=1.$$
 The explicit value of $K_\d$ can be easily derived from Mehler's formula.
 Indeed, assuming for simplicity that $\Omega=\R^N$ equipped with
 $\P=\exp{-(\sum x_j^2/2)}  dx_1\cdots dx_N (2\pi)^{-N/2}$ and that $\{g_n\mid 1\le n\le N\}$
 are the coordinates on $\R^N$,
 we have
 $$(T_\d F) (x)= \int F(\d x+ (1-\d^2)^{1/2} y) \P(dy)
 = (2\pi (1-\d^2))^{-N/2}\int F( t)  \exp{-(\frac{|t-\d x|_2^2}{2(1-\d^2)} )}  dt_1\cdots dt_N$$
 from which we derive
 $$K(x,t  )=  (1-\d^2)^{-N/2}\exp{-(\frac{|t-\d x|_2^2}{2(1-\d^2)} )} \exp{(|t|_2^2/2)} $$
 and finally
 $$K(x,t  )= (1-\d^2)^{-N/2}\exp{  \frac{-\d^2 |t|_2^2 +2\d t.x   -\d^2 |x|_2^2}{2(1-\d^2)} }.$$
 
 We will invoke the following simple fact.
 \begin{lem}\label{cm}
 For any $z=(z_n)\in [-1,1]^N$ there is a positive operator
 $\Theta_z: \ L_1(\P) \to L_1(\P) $ of norm $1$ such that
 $$\forall  n=1,\cdots, N\quad \Theta_z(g_n)= z_n g_n.$$
 \end{lem}
 \begin{proof}  Let $T^{(1)}_\d$
 be the operator  corresponding  to $T_\d$ in the case $N=1$.
 Then we simply may take
 $$\Theta_z=T^{(1)}_{z_1}
\otimes  \cdots \otimes T^{(1)}_{z_N}.$$ \end{proof}
\section{Orlicz spaces of subgaussian variables}

We now turn to the behaviour of subgaussian variables in $L_p$ for $p<\infty$.
We start by recalling the definition of certain Orlicz spaces.
The latter  are analogues of the $L_p$-spaces
obtained when one replaces the function $x\mapsto x^p$
by a more general convex increasing function $\psi:\ \R+ \to \R+$
such that $\psi(0)=0$.\\
Let $(\Omega,m)$ be a measure space.
The Orlicz space $L_\psi(\Omega,m)$ (or $L_\psi(m)$,  or simply $L_\psi$) is the space of those $f\in L_0  (\Omega,m)$ for which there is $t>0$ such that
 ${\bb E} \psi (|f|/t)< \infty$ and we set
\[
\|f\|_{{\psi} }  = \inf\{t> 0\mid {\bb E}\psi (|f|/t)\le \psi(1)\}.
\]
It is known that the resulting space is a Banach space
and, if $m$ is finite, we have $L_\infty\subset L_\psi\subset L_1$.

We will be interested by the particular case of exponentially growing functions,
so we limit our discussion to that special case.
 Let $0<a<\infty$.
Let $$\forall x>0\quad \psi_a(x)=\exp {x^a}-1.$$ 
Then
\[
\|f\|_{{\psi_a} }  = \inf\{t> 0\mid {\bb E}\exp|f/t|^a\le e\}.
\]
In many cases the growth of
 the $L_p$-norms of a function when $p\to \infty$ is equivalent  to its exponential integrability, as in
the following elementary and well known Lemma.

\begin{lem}\label{41}
Fix a number $a>0$. The following properties of a 
(real or complex) random variable $f$ are 
equivalent:
\begin{itemize}
\item[(i)] $f\in L_p$ for all $p<\infty$ and $\sup\nolimits_{p\ge 1} p^{-1/a}\|f\|_p<\infty$.
\item[(ii)]  $f\in {L_{\psi_a} }$.
\item[(iii)] There is $t>0$ such that 
$\sup\n_{c>0} \exp{(tc^a)} \P\{ |f|>c\} <\infty$. 
\item[(iv)] Let $(f_n)$ be an i.i.d. sequence of copies of $f$. Then
$$\sup\n_n (\log(n+1))^{-1/a} |f_n| <\infty \text{   a.s.  } .$$ 
\end{itemize}
Moreover,  there is a positive constant $C_a$ such that for any $f\ge 0$ we have
 \begin{equation}\label{73}
C_a^{-1} \sup\nolimits_{p\ge 1} p^{-1/a}\|f\|_p \le \|f\|_{{\psi_a} } \le C_a \sup\nolimits_{p\ge 1} 
p^{-1/a} \|f\|_p,
\end{equation}
and we can restrict the sup over $p$ to be over all even integers.
\end{lem}

\begin{proof}
Assume that the supremum in (i) is $\le 1$. Then
\begin{align*}
{\bb E} \exp|f/t|^a &= 1 + \sum\nolimits^\infty_1 {\bb E}|f/t|^{an} (n!)^{-1} 
\le 1 + \sum\nolimits^\infty_1 (an)^n t^{-an}(n!)^{-1}\\
\intertext{hence by Stirling's formula for some constant $C$}
{\bb E} \exp|f/t|^a &\le 1 + C \sum\nolimits^\infty_1(an)^n t^{-an} n^{-n}e^n=
 1 + C \sum\nolimits^\infty_1(at^{-a}e)^n
\end{align*}
from which it becomes clear (since $1<e$) that (i) implies (ii). Conversely, if (ii) holds we 
have a fortiori for all $n\ge 1$
\[
(n!)^{-1} \|f/t\|^{an}_{an} \le {\bb E} \exp|f/t|^a \le e
\]
and hence
\[
\|f\|_{an} \le e^{\frac1{an}}  (n!)^{\frac1{an}} t \le e^{\frac1{a}} n^{\frac1a} t = 
 (an)^{\frac1a} t(e/a)^{1/a},
\]
which gives $\|f\|_p \le   p^{1/a}t(e/a)^{1/a}$ for the values $p=an$, 
$n=1,2,\ldots$~. One can then easily interpolate (using H\"older's inequality)  to obtain (i). The equivalences of (ii) with (iii) and (iv) are
 elementary exercises.
The last 
assertion is   a simple recapitulation left to the reader.
\end{proof}
 The following variant explains 
why the variables with $ \|f\|_{ {\psi_2}}<\infty$ are sometimes called subgaussian.

\begin{lem}\label{42} Let $f\in L_1(\Omega,\P)$ be real valued such that
  $\E f=0$. Then
 $f\in L_{\psi_2} $ iff $f$ is subgaussian.\\
Moreover, $\|f\|_{{{\psi_2} }}$, $\sup\n_{p\ge 1} p^{-1/2}\|f\|_p$ and
$sg(f)$ are equivalent quantities for such $f$'s.
\end{lem}
\begin{proof} Assume that $f\in L_{\psi_2} $ with $\|f\|_{{\psi_2}}\le 1$.
Let $f'$ be an independent copy of $f$. Let $F=f-f'$. Note that since
the distribution of  $F$ is 
  symmetric all its odd moments vanish, and hence
  $$\E \exp{xF} =1+\sum\n_{n\ge 1}  \frac{x^{2n}}{2n !} \E F^{2n}.$$
 We have $\|F\|_{\psi_2}\le  \|f\|_{\psi_2} + \|f'\|_{\psi_2}\le 2$. Therefore
 $ \E (F/2)^{2n} \le n!  \E \exp{(F /2)^2} \le  e n! $. Therefore
 $$\E \exp{xF} \le 1+\sum\n_{n\ge 1}  \frac{(2x)^{2n}}{2n !} e n!\le 1+\sum\n_{n\ge 1}  \frac{(2\sqrt{e}x)^{2n}}{n !} \le \exp {(4ex^2)}. $$
 But since $t\mapsto \exp -xt$ is convex for any $x\in \R$,
 and $\E f'=0$
 we have $1=e^0 \le \E\exp -x f'$ and hence
 $\E\exp xF=\E\exp xf \E\exp - xf'\ge \E\exp xf $.
 Thus we conclude $sg(f)\le (8e)^{1/2}$. By homogeneity this
 shows $sg(f)\le (8e)^{1/2} \|f\|_{\psi_2}$.\\
 Conversely, assume $sg(f)\le 1$. Then 
 by \eqref{102} and \eqref{103}    
 $$\P(\{ |f|>t\})  \le 2 e^{-t^2/2}.$$
 Fix $c>\sqrt 2$. Let $\theta= 1/2-1/c^2$. Note $\theta>0$. We have
 $$\E \exp{(f/c)^2}-1  =\int_0^\infty   (2t/c^2) \exp{(t/c)^2} \P(\{ |f|>t\})   dt
 \le   \int_0^\infty   (4t/c^2)   e^{-\theta t^2}   dt= 2/\theta c^2  .$$
 Elementary calculation shows that if $c=({2}(e+1)(e-1)^{-1})^{1/2}$ we have
 $1+2/\theta c^2=e$. Thus we conclude
 $\|f\|_{\psi_2}\le ({2}(e+1)(e-1)^{-1})^{1/2}$.
By homogeneity, this
 shows\\ $  \|f\|_{\psi_2} \le ({2}(e+1)(e-1)^{-1})^{1/2} sg(f)$.
Lastly the equivalence between $  \|f\|_{\psi_2} $
and 
$\sup\n_{p\ge 1} p^{-1/2}\|f\|_p$
is a particular case of \eqref{73}. 
\end{proof}
The equivalence between (ii) and (iv) of Lemma \ref{41} can be made more precise,
as follows.
\begin{lem}\label{leeq}
The norm $f\mapsto \|f\|_{\psi_a}$
on $L_{\psi_a}$ is equivalent to
$f\mapsto \E\sup\n_{n\ge 1}  (\log(n+1))^{-1/a} |f_n| $.
\end{lem}
\begin{proof} Assume $\| f\|_{\psi_a}\le 1$.
Then $\E \exp{|f|^a} \le e$. Let $F=\sup\n_{n\ge 1} (\log(n+1))^{-1/a} |f_n| $.
Then $\forall c>0$
$$\P(\{F>c\})\le \sum\n_1^\infty \P(\{|f|>c(\log(n+1))^{-1/a}\})
\le  \sum\n_1^\infty e \exp{(-c^a\log(n+1)}=e \sum\n_1^\infty (n+1)^{-c^a}.$$
If $c^a>4$ we have a fortiori
$$\P(\{F>c\})\le e \sum\n_1^\infty (n+1)^{-2} 2^{-c^a/2}\le K 2^{-c^a/2},$$
where $K=e \sum\n_1^\infty (n+1)^{-2}$. From this we derive immediately
$$\E F=\int_0^\infty  \P(\{F>c\})dc \le K'$$
where $K'=4^{1/a} +\int_{4^{1/a}} K 2^{-c^a/2}dc$.
By homogeneity, this yields $\E F\le K' \|f\|_{{\psi_a}}$ for any $f\in L_{\psi_a}$.\\
We now turn to the converse. Assume $\E F\le 1$.
Then $\P(\{F\le 2\})
\ge 1/2$, and hence
$$\prod\n_{n\ge 1} \P(\{|f|\le 2 (\log(n+1))^{1/a}\})  \ge 1/2.$$
But $\P(\{|f|\le 2 (\log(n+1))^{1/a}\})=1-\P(\{|f|> 2 (\log(n+1))^{1/a}\})\le \exp^{-\P(\{|f|> 2 (\log(n+1))^{-1/a}\})}$ and hence
$$\sum\n_{n\ge 1} \P(\{|f|> 2 (\log(n+1))^{-1/a}\}) \le \log 2$$
or equivalently
$$\sum\n_{n\ge 1} \P(\{\psi_a(|f|/2) >n   \}) \le \log 2.$$
But it is classical that for any variable $Z\in L_1$ we have
$\E Z\le 1+\sum\n_{n\ge 1} \P(\{Z>n\})$, so we conclude
$$\E \psi_a(|f|/2) \le 1+\log 2 \le e,$$
and hence $\|f\|_{{\psi_a}}\le 2$. 
By homogeneity, 
$\|f\|_{{\psi_a}}\le 2\E F$ for any $f\in L_{\psi_a}$.
\end{proof}
\begin{rem}[On $L_{\psi_a}$ and the Fourier transform]\label{p43} Let $G$ be  a compact Abelian group. Let $f_1,f_2, g_1,g_2\in L_4(G)$.
 It is well known that if $|\hat f_j |\le \hat g_j$ on $\hat G$ ($j=1,2$)
 then
 $$\|f_1 f_2\|_2 \le \|g_1  g_2\|_2.$$
 Indeed, this follows from
 $\|f_1 f_2\|_2=\|\hat f_1\ast \hat f_2\|_2$, $|\hat f_1\ast \hat f_2|\le \hat g_1\ast\hat g_2$  and again $\|g_1  g_2\|_2=\|\hat g_1\ast\hat g_2\|_2$.\\
 Iterating this idea, we find that if $f_1,\cdots,f_m\in L_{2m}(G)$
 are such that $|\hat f_j |\le \hat g_j$ on $\hat G$ ($j=1,\cdots, m$) we have
 $$\|f_1  \cdots    f_m\|_2 \le \|g_1  \cdots   g_m\|_2.$$
 In particular, taking $f_1=\cdots  =f_m=f$ and $g_1=\cdots  =g_m=g$
 we find that if $f,g\in L_{2m}(G)$ are such that $|\hat f|\le \hat g$ on $\hat G$,
 then
 $$\|f\|_{2m} \le \|g\|_{2m}.$$
 This implies that for any $a>0$ we have
 $$ \sup\nolimits_{p\in 2 \N} p^{-1/a}\|f\|_p\le \sup\nolimits_{p\in 2 \N} p^{-1/a}\|g\|_p .$$
By \eqref{73}, we have
$$\|f\|_{\psi_a}  \le C_a\|g\|_{\psi_a} ,$$
where $C_a$ is a constant depending only on $a$.
\end{rem}

\section{Slepian's and Talagrand's Comparison Theorems}

A collection of random variable $\{X_s\mid s\in S\}$
 on a probability space $(\Omega,\P)$ is called
 Gaussian (resp. subgaussian)
if all the variables in its linear span
are Gaussian (resp. subgaussian). In this definition,
we include in parallel
the real and complex case, that we will distinguish
if necessary by $\R$-Gaussian or $\C$-Gaussian
(resp. $\R$-subgaussian or $\C$-subgaussian).

\noindent{\bf Convention:} To avoid any discussion concerning separability of random processes, 
for any real valued random process $\{X_s\mid s\in S\}$ in $L_1(\Omega,\P)$
by convention,  we define the number $\E \sup\n_{s\in S} X_s$
(possibly $=\infty$)  by
setting
$$\E \sup\n_{s\in S} X_s =\sup\n_{S'\subset S} \E\sup\n_{s\in S'} X_s,$$
where the sup runs over all {\it finite} subsets ${S'\subset S}$.

The following comparison theorem originally due to Slepian
is of paramount importance in the theory of Gaussian processes. It
was later on refined by various authors. The version we state
was popularized by Fernique (see \cite{Fe2}).
\begin{thm}[Slepian's comparison principle] Let $\{X_s\mid s\in S\}$ and $\{Y_s\mid s\in S\}$
be two $\R$-Gaussian processes such that
$$\forall s,t\in S\quad \|Y_s-Y_t\|_2 \le \|X_s-X_t\|_2.$$
Then
$$\E\sup\n_{s\in S} Y_s \le \E\sup\n_{s\in S} X_s.$$
Moreover if we also have $\E|Y_s|^2=\E|X_s|^2$ for all $s\in S$
then for any finite $S'\subset S$
$$\forall c\in \R \quad
\P(\{   \sup\n_{s\in S'} Y_s>c   \})   \le 
\P(\{   \sup\n_{s\in S'} X_s>c   \}).$$
\end{thm}

We should emphasize that this is a quite non trivial phenomenon, special to Gaussian processes. Indeed,
in general a comparison of the covariances
is far from implying a comparison of the suprema of the processes.

It is natural to wonder whether a similar comparison theorem
holds when $Y$ is merely subgaussian.
This turns out to be true, but highly non trivial:

\begin{thm}[Talagrand's comparison principle]\label{tal}
Let $\{X_s\mid s\in S\}$ 
be   $\R$-Gaussian process
and  $\{Y_s\mid s\in S\}$
  $\R$-subgaussian.
Assume
$$\forall s,t\in S\quad sg(Y_s-Y_t)   \le \|X_s-X_t\|_2,$$
or equivalently
$$\forall x\in \R\ \forall s,t\in S\quad \E \exp {x(Y_s-Y_t)} \le \exp{ (x^2\E|X_s-X_t|^2/2)}.$$
Then
$$\E\sup\n_{s\in S} Y_s \le \tau \E\sup\n_{s\in S} X_s,$$
where $\tau$ is a numerical constant.
\end{thm}
The genesis of this result started when Fernique (see \cite{Fe2})
proved his characterization of stationary Gaussian processes
with a.s. bounded sample paths.
His result implied that if $S$ is a group and if the distribution
of  $\{X_s\mid s\in S\}$ is invariant under translation (stationarity),
then the comparison in Theorem \ref{tal} holds
for any $\R$-subgaussian $\{Y_s\mid s\in S\}$. Later on,
Talagrand proved a similar characterization
(the so-called majorizing measure condition)
of  Gaussian processes
with a.s. bounded sample paths, without assuming any stationarity.
To explain this, let us go back to the stationary case.
Roughly, when $S$ is a compact group and
$\{X_s\mid s\in S\}$ is stationary the normalized Haar measure
on $S$ provides a way to estimate $\E\sup\n_{s\in S} X_s$.
More precisely, $\E\sup\n_{s\in S} X_s$
is equivalent to the metric entropy integral 
$$  \int\n_0^\infty (\log N_X(\vp))^{1/2} d\vp, $$
where $N_X(\vp)$ is the smallest number of a covering
of $S$ by open balls of radius $\vp$ for the metric
$d_X(s,t)=(\EE|X_t-X_s|^2)^{1/2}$. (Note that $\log N_X(\vp)=0$
when $\vp$ is larger than the diameter, and the latter is necessarily finite).
In the stationary case,
when both the Haar measure and $d_X$ are translation invariant,
$N_X(\vp)$ is equivalent to ${m_G(\{s\mid d_X(s,1)<\vp\})}^{-1}$ and hence
the latter integral
is equivalent to
$${\cl I}_2(X)=\int_0^\infty  (\log \frac{1}{m_G(\{s\mid d_X(s,1)<\vp\})})^{1/2} d\vp.$$
When ${\cl I}_2(X)<\infty$ it is known (this is a subgaussian variant of Dudley's majorization of
  Gaussian processes) that all the $\R$-subgaussian processes 
$\{Y_s\mid s\in S\}$
 such that $sg(Y_s-Y_t)\le d_X(s,t)$ satisfy
 $$\E\sup\n_{s\in S} Y_s \le \tau' {\cl I}_2(X)$$
 for some numerical constant $\tau'$.
 Together with the equivalence $\E\sup\n_{s\in S} X_s\simeq {\cl I}_2(X)$
 this leads to Theorem \ref{tal} assuming $X$ \emph{stationary} $\R$-Gaussian.
 
 For general a.s. bounded Gaussian processes $(X_t)$,
 Fernique conjectured  the existence of a ``majorizing measure"
 that would replace Haar measure.
 Namely there should exist
  a probability $\mu$ on $S$
 such that
 \begin{equation}\label{201}
 {\cl I}(\mu,X) =\sup\n_{t\in S} \int_0^\infty (\log \frac{1}{\mu(\{s\mid   d_X(s,t)<\vp\})})^{1/2} d\vp<\infty .\end{equation}
 More precisely, 
 for some constant $c>0$,
  we should have  for any bounded Gaussian processes $(X_t)$
  \begin{equation}\label{z201} \inf_{\mu} {\cl I}(\mu,X) \le  c \E\sup\n_{s\in S} X_s\end{equation}
  where the infimum
  on the left-hand side runs over all probabilities $\mu$ on $S$.
  In the latter form, the question can be reduced
  to the case when $S$ is a finite set (with-of course-$c$ independent of $S$).
  In his   paper \cite{Ta} (see also \cite[\S 2.4]{Ta2})
 Talagrand proved this conjecture. This was a major achievement. 
He showed
that if $\E\sup\n_{s\in S} X_s\le 1$ there is a probability measure $\mu$ 
(the so-called majorizing measure) 
satisfying \eqref{201}. Here again \eqref{201}
also allows one to majorize all the $\R$-subgaussian processes 
$\{Y_s\mid s\in S\}$
 such that $sg(Y_s-Y_t)\le d_X(s,t)$, namely  
 we have a numerical constant
 such that $\E\sup\n_{s\in S} Y_s \le \tau'' {\cl I}(\mu,X)$. Thus he
 obtains Theorem \ref{tal} as a corollary of 
 his main result, just like in the stationary case.
 Note that, even though it does not involve
 majorizing measures,  as far as we  know the only   known proof
  of Theorem \ref{tal}  uses \eqref{z201} in some form or other.
  In later work Talagrand chose  to reformulate the majorizing
  measure condition in terms of what he called
    chainings, and he emphasized the ``generic chaining" :
    he showed that the quantity $\inf_{\mu} {\cl I}(\mu,X) $
    that is equivalent (with universal constants independent of $X$ or $S$)
     to $\E\sup\n_{s\in S} X_s$ is similarly equivalent
     to
     $$\inf \sup\nolimits_{s\in S}  \sum\nolimits_{n\ge 0}  2^{n/2} d_X(s,S_n)$$
     where the infimum is now taken over
     all sequences of subsets $S_n\subset S$
     with cardinality $|S_n|<2^{2^n}$ for all $n$. See \cite{Ta3,Ta2}.
  \begin{rem}\label{osc}
  For any $\R$-Gaussian process (or any real valued process such that $\{X_s\mid s\in S\}$
  and $\{-X_s\mid s\in S\}$ have the same distribution)
  we have
  $$\E\sup\n_{s\in S} X_s =\E \sup\n_{s,t\in S} |X_s-X_t|/2.$$
  Indeed,
  $\E \sup\n_{s,t\in S} |X_s-X_t|=\E \sup\n_{s,t\in S} (X_s-X_t)=
  \E \sup\n_{s } X_s+  \E \sup\n_{t\in S} -X_t=2\E\sup\n_{s\in S} X_s$.
  \end{rem}
 \begin{cor} Let $(f_n)$ be a (real or complex)
 subgaussian sequence with $sg(\{f_n\})\le 1$.
 Let $(g_n)$ be a normalized i.i.d. $\R$-Gaussian sequence.
   Let $E_g$ (resp. $E_f$) be the linear span of $(g_n)$ (resp. $(f_n)$).
   Let $u:\ E_g \to E_f$ be the linear operator such that
   $u(g_n)=f_n$.
   Then for any $n$ and any  $x_1,\cdots, x_n\in E_g$ we have
    \begin{equation}\label{301}\E \sup\n_j |u(x_j)|\le C_0 \E  \sup\n_j |x_j|,\end{equation}
   where $C_0$ is a numerical constant. 
  \end{cor}
   \begin{proof} 
   Assume  first that $(f_n)$ is  $\R$-subgaussian and $sg(\{f_n\})\le 1$.
   Assume the linear spans and $u$ are all in the $\R$-linear sense.
   Let $y_j=u(x_j)$. Then, since  $sg(\{f_n\})\le 1$,
   for any $1\le s,t\le n$ we have
   $sg(y_s-y_t)\le \|x_s-x_t\|_2$.
   Also $sg(y_s )\le \|x_s \|_2$.
   A fortiori (see Lemma \ref{im})  we have $\|y_s \|_2\le \|x_s \|_2$.
By Theorem \ref{tal} with $S=\{1,\cdots,n\}$ we have
   $\E\sup y_s \le \tau \E\sup x_s$, and also
   $\E\sup -y_s \le \tau \E\sup x_s$. Therefore 
   $\E\sup\n_{s,t\in S} |y_s-y_t|=\E\sup\n_{s,t\in S} y_s-y_t
   \le 2 \tau   \E\sup x_s
   \le 2\tau \E\sup |x_s|
   $, and hence
   $$\E\sup\n_{s\in S} |y_s |\le \E|y_1|+\E\sup\n_{s\in S} |y_s-y_1 |
   \le \|y_1 \|_2+2\tau \E\sup |x_s|\le \|x_1 \|_2+2\tau \E\sup |x_s|$$
   and since $ \|x_1 \|_2\le  (2/\sqrt \pi)\|x_1 \|_1$ we obtain the announced
   result with $C_0\le 2/\sqrt \pi +2\tau$.\\
   Now assume $(f_n)$ is  $\C$-subgaussian but
 with $u,E_g,E_f$ still with respect to $\R$-linearity, the first part of the proof can be applied separately to
 the real and imaginary parts of $(f_n)$, then the triangle inequality
 yields \eqref{301} with a double constant.
 Lastly, if $E_g$ is the $\C$-linear span and
 $u$ is  $\C$-linear, if $x=\sum (a_k+i b_k) g_k$
 we have $u(x)= u(\sum a_k  g_k) + i u( \sum b_k g_k)$
 and hence $| u(x)|\le |u(\sum a_k  g_k)|+ |u( \sum b_k g_k)|$
 and again the first part of the proof allows us to conclude
 that  \eqref{301} holds. 
    \end{proof}
  
\def\p{\cl P}
\def\g{\gamma}
 We will need one more characterization
 of subgaussian sequences,
 for which the next definition
will be very useful.

 \begin{dfn}\label{dom} Consider families $\{\varphi_n\}\subset L_1(T,m)$,
 and   $\{\g_n\}\subset L_1(T',m')$ indexed by the same index set $I$.
 We   say
 that $(\varphi_n)$ is $C$-dominated by $(\g_n)$ if  
 \begin{equation}\label{1nb} \text{there is an operator }   u:\ L_1(m')\to L_1(m)   \text{ with }
 \|u\|\le C \text{ such  that }\\
  \ u(\g_n)=\varphi_n\ \forall n\in I .
 \end{equation}
 \end{dfn}
 
 \begin{pro}[\cite{Lev}, see also \cite{Pi33}]\label{r42} The sequence
 $(\varphi_n)$ is $C$-dominated by $(\g_n)$ iff
 for any $N$ and any $f_1 ,\cdots,f_N$ in  the linear span
 of $\{  \g_n \}$ of the form $f_i=\sum\n_j a_{ij} \g_j$, the associated
 $\tilde f_i=\sum\n_j a_{ij} \varphi_j$ satisfy
\begin{equation}\label{21}  \| \sup\n_i |\tilde f_i| \|_1 \le C \| \sup\n_i |  f_i| \|_1 .\end{equation}
 \end{pro}
  \begin{proof} 
  Let $E$ be the  linear span
 of $\{  \g_n \}$. Assume \eqref{21}. Our assumption implies a fortiori that
 $\| \sum\n_j a_{j} \varphi_j\|_1\le \| \sum\n_j a_{j} \g_j\|_1$. Therefore
  we can unambiguously define
 $u:\ E \to L_1(T',m')$, by setting $u(\sum\n_j a_{j} \g_j)=
 \sum\n_j a_{j} \varphi_j$. Our assumption then means
 that $\| \sup|u(f_i)|  \|_1 \le  C \| \sup|f_i|  \|_1 $
 for any finite set $(f_i)$ in $E$. The content of the Proposition
 is that $u$ admits an extension $\tilde u: \ L_1(m)\to L_1(m')$
 with $\|\tilde u\|\le C$.
  We will reduce the proof
  to the following claim. Assume that $(T',m')$ is an atomic measure space
  and that $T'$ is partitioned into a finite set
  of disjoint atoms $A_1,\cdots,A_n$.
  If for any $f_1,\cdots,f_n\in E$ we have
  $$|\sum\n_1^n  \int_{A_i} u(f_i) dm'  |\le C \| \sup|f_i|  \|_1 $$
  then $u$ admits an extension $\tilde u: \ L_1(m)\to L_1(m')$
 with $\|\tilde u\|\le C$.\\
 Let us first accept this claim.
 Note that $|\sum\n_1^n  \int_{A_i} u(f_i) dm'  |\le \| \sup|u(f_i)|  \|_1$.
 Thus the claim is nothing but the Proposition in the case
 when $(T',m')$ is atomic with finitely many atoms.
 Thus using the directed net of finite subalgebras of $(T',m')$
 one can get an extension $\tilde u:\ L_1(T,m)\to L_1(T',m')^{**}$
 with norm $\le C$,
 and then, using  the fact that there is a projection of norm $1$
  from $L_1(T',m')^{**}$ to $L_1(T',m')$ (see Remark \ref{r20}), we get a $\tilde u$
  with range into $L_1(T',m')$. Thus it suffices to check the claim.
  This is an application of Hahn-Banach.
  Let $\cl E=E^n$  equipped with the norm induced by 
  $L_1(m; \ell^\infty_n)$, or more explicitly for all $f=(f_1,\cdots,f_n)\in \cl E$
  we set 
  $\|(f_1,\cdots,f_n)\|_{\cl E} =\| \sup|f_i|  \|_1$.
  Let $\xi\in \cl E^*$ be the linear form defined   for all $f \in \cl E$
 by
  $$\xi (f)= \sum\n_1^n  \int_{A_i} u(f_i) dm' .$$
  By our assumption $\|\xi\|_{\cl E^*}\le C$.
  Let $\tilde \xi \in L_1(m; \ell^\infty_n)^*$ be the Hahn-Banach
  extension of $\xi$, such that for all $F=(F_1,\cdots,F_n)\in L_1(m)^n$
    $$|\tilde \xi (F)|\le C \| \sup|F_i|  \|_1.$$
Obviously we have $\Phi_1,\cdots,\Phi_n$ in $L_\infty(m)$
such that $\| \sum   | \Phi_i |\|_\infty \le C$ and 
such that $\tilde \xi (F)=\sum \int \Phi_i F_i dm$. 
Note   that  for any $f\in \cl E$ we have
$\sum \int \Phi_i f_i dm= \xi(f)=  \sum\n_1^n  \int_{A_i} u(f_i) dm' .$
Let then
    $\tilde u(x)= \sum\n_{i} 1_{A_i} m'(A_i)^{-1} (\int \Phi_i x dm) $.
    Clearly $\|\tilde u\|\le \| \sum  | \Phi_i |\|_\infty \le C$,
    and (recalling that $u(x)\in \text{span}[1_{A_i}]$) we have
    $$\forall x\in E\quad \tilde u(x)=\sum\n_{i} 1_{A_i} m'(A_i)^{-1}\int_{A_i} u(x) dm'=  \sum\n_{i} 1_{A_i}u(x)  =u(x).$$  This proves the claim.
 \end{proof}
 \begin{rem}\label{r20} Let $(T, {\cl A},m)$ be a countably generated probability space,
 so that there is an increasing filtration $({\cl A}_n)$ of finite 
 $\sigma$-subalgebras whose union generate ${\cl A}$.
 The classical fact that there is a norm $1$-projection   $P:\ L_1(T,m)^{**}\to L_1(T,m)$ is easy to prove using martingales as follows.
 Just observe that any $f\in L_1(T,{\cl A},m)^{**}=L_\infty(T,{\cl A}, m)^{*}$
 induces by restriction to $L_\infty(T,{\cl A}_n,m)$ a sequence $(f_n)$ with
 $f_n\in L_1(T,{\cl A}_n,m) =L_\infty(T,{\cl A}_n, m)^{*}$. It is easy to see that
 $(f_n)$  is a martingale bounded in $L_1(T,{\cl A},m)$ by the norm
 of $f$ in $L_1(T,{\cl A},m)^{**}$.
 By the martingale convergence theorem, $(f_n)$ converges a.s.
 to a limit $f_\infty\in L_1(T,{\cl A},m)$, with $\|f_\infty  \|_1\le \|f\|$.
 In general the convergence does not hold in $L_1(T,{\cl A},m)$.  However if our original 
 $f\in L_1(T,{\cl A},m)^{**}$ happens to be in $L_1(T,{\cl A},m)$
 then the convergence holds in $L_1(T,{\cl A},m)$ and $f_\infty=f$.
 Thus if we set $P(f)=f_\infty$, we obtain the desired projection.
 See our recent book \cite{Pi4} for basic  martingale convergence theorems and
for more information of the many connections of martingale theory with
Banach space theory and
harmonic analysis.
  \end{rem}
 We denote by $(g_n)$ an i.i.d. sequence of normalized
 $\R$-Gaussian  random variables on some probability space
 $(\Omega,\P)$. \\
 Given a sequence $\{\varphi_n\}\subset L_1(T,m)$,
 we denote by $\{\varphi_{n,k}\}\subset L_1(T^\N,m^{\otimes \N})$ the family
 defined by
 $$ \forall t\in T^\N\quad \varphi_{n,k} (t)= \varphi_{n}(t_k).$$
 
 Note that the definition   of subgaussian (Definition \ref{40b})   shows that if $(\varphi_n)$ is subgaussian,
 $\{\varphi_{n,k}\}$ is also subgaussian with $sg(\{\varphi_{n,k}\})=sg(\{\varphi_{n}\})$.
 
Concerning Definition \ref{dom}: we will need to consider $(\varphi_n)$ 
such that   $\{\varphi_{n,k}\}$
 is $C$-dominated by $(g_{n,k})$. Of course
 the reader will note that the sequences $(g_{n,k})$ and $(g_{n})$ have the same distribution, so we will say (abusively) in this case that   $\{\varphi_{n,k}\}$
 is $C$-dominated by $(g_n)$.  
 \\
 We will denote by $C_{dom}(\{\varphi_{n}\}) $  the smallest    $C$
 such that $\{\varphi_{n}\}$
 is $C$-dominated by $(g_n)$.  

   \begin{pro}\label{p8}
There is a numerical constant $c_1$ such that 
any $C$-subgaussian sequence $\{\varphi_n\}\subset L_1(T,m)$
is $c_1C$-dominated by $(g_n)$. \\ More precisely, 
assuming $\E \varphi_n=0$ for all $n$, the following are equivalent.
\begin{itemize}
\item[(i)]  For some $C$ the sequence $\{\varphi_n\}\subset L_1(T,m)$ is $C$-subgaussian.
\item[(ii)] For some $C'$ the sequence $\{\varphi_{n,k}\}$ is $C'$-dominated by $(g_n)$. 
 \end{itemize}
 Moreover, we have
 $$c_1^{-1} C_{dom}(\{\varphi_{n,k}\} ) \le  {sg}(\{\varphi_n\}) \le c_2 C_{dom}(\{\varphi_{n,k}\})$$
 where $ c_2$ is another positive constant independent of $\{\varphi_n\}$.
  \end{pro}
 \begin{proof}[Sketch]
 The first assertion is a consequence of  Talagrand's comparison principle
 together with Proposition \ref{r42}.
 From this we deduce $C_{dom}(\{\varphi_{n}\} )   \le c_1{sg}(\{\varphi_n\})$. 
 As we already observed,
 ${sg}(\{\varphi_n\})$ is equal to
 ${sg}(\{\varphi_{n,k}\}) $. 
 Thus $C_{dom}(\{\varphi_{n,k}\})   \le c_1{sg}(\{\varphi_{n,k}\}) =c_1{sg}(\{\varphi_n\}) $, and hence (i) $\Rightarrow$ (ii).\\
 Conversely, if (ii) holds,
 for any $f=\sum x_n \varphi_n$
 with $\sum |x_n|^2=1$ we have (with the notation in Lemma \ref{leeq})
 $$\E \sup\n_{k\ge 1} (\log(k+1))^{-1/2} |f_k|
 \le C_{dom}(\{\varphi_{n,k}\}) \E \sup\n_{k\ge 1} (\log(k+1))^{-1/2} |g_k|$$
 and hence by Lemma \ref{leeq}
 $\|f\|_{\psi_2}\le c'_2 C_{dom}(\{\varphi_{n,k}\})$
 for some numerical constant $c'_2$.
 By Lemma \ref{42}
 we obtain $sg(f)\le c_2  C_{dom}(\{\varphi_{n,k}\})$
 for some numerical constant $c_2$,
 or equivalently $sg(\{\varphi_n\})\le c_2  C_{dom}(\{\varphi_{n,k}\})$, which proves  (ii) $\Rightarrow$ (i).
\end{proof}

\section{Subgaussian sequences in harmonic analysis}

More subgaussian examples come from Fourier analysis.
Let $0<k(0)<k(1)<\cdots<k(n)<\cdots$ be a sequence of integers
 such that
\begin{equation}\label{h01}\inf\n_{n} \{k(n+1)/k(n)\}>1.\end{equation}
Such sequences are called ``Hadamard lacunary".
The simplest example is the sequence $k(n)=2^n$.
The associated  sequence
$$f_n=\exp{(i k(n) t)}$$
on $[0,2\pi],dt/2\pi$ is subgaussian. 
 We will check this in Proposition
\ref{pqi}.
Of course
the real (or the imaginary) parts
also form a subgaussian sequence.
 Although these are not independent   random variables
on the unit circle, it turns out that they behave in many ways
as independent ones. For instance, while the sequence
$f_n(t)=\sin(2^nt)$ is not independent,   the 
$\pm1$-valued
sequence
formed of its signs $({\rm sign}( f_n(t)))$ is stochastically independent.

For any subset  $\Lambda \subset \Z$ not containing $0$
we   say that  $\Lambda $ 
is subgaussian if the system
$$\tilde\Lambda=\{\exp{(i k  t)}\mid k\in \Lambda\}$$
is subgaussian. We set by convention
$$sg(\Lambda)=sg(\tilde\Lambda).$$

More generally we will consider subsets $\Lambda$
of a discrete Abelian group $\hat G$. Then $\Lambda$ is formed
of continuous characters on the dual group $G$, which is a compact Abelian group equipped with its normalized Haar measure $m_G$.
In that case $sg(\Lambda)$ is the subgaussian constant
of the family $\{\gamma\mid \gamma\in \Lambda\}$ viewed
as random variables on $(G,m_G)$.

The sequence $\{2^n\}$ is close to independent
in the following sense:
\begin{dfn} A subset $\Lambda\subset \Z$ is called quasi-independent
if  the   sums $\sum\n_{n\in A} n$ are distinct integers
when $A$ runs over all the finite subsets of $\Lambda$.
\end{dfn} 
\begin{rem}\label{r10} Any sequence $\{k(n)\}$ 
such that
$k(n)>\sum\n_{j<n} k(j)$
(for example $k(n)=2^n$) is clearly quasi-independent.
\end{rem} 
A finite set $\Lambda\subset \Z$ is quasi-independent iff
$$\int \prod\n_{n\in \Lambda} (1+ e^{int}+ e^{-int}) dt/2\pi=1,$$
or equivalently iff for some $0<\d\le 1$
$$\int \prod\n_{n\in \Lambda} (1+ \d(e^{int}+ e^{-int})) dt/2\pi=1.$$
Indeed, the preceding integral can be rewritten as
$1+\d a_1+\d^2 a_2+\cdots+\d^{|\Lambda|}  a_{|\Lambda|}$
where $a_1,a_2,\cdots$ are non-negative integers.\\
From now on let $dm(t)=dt/2\pi$ on $[0,2\pi]$.
We have then
\begin{pro}\label{pqi} 
Any quasi-independent sequence $\Lambda\subset \Z$ 
is  subgaussian on $([0,1],dt/2\pi)$ with constant $\le 2$. More generally,
any Hadamard lacunary sequence is  subgaussian.
\end{pro}
\begin{proof} We may assume $\Lambda$ finite and $0\not\in \Lambda$
 For any $z=(z_k)\in \T^\Lambda$
 let  
 $$F_z=  \prod\n_{n\in \Lambda} (1+ \Re(\bar z_n e^{int}) ) .$$
 Note that if $k=\sum\n_{n\in A} n$ we have $\hat {F_z}(k)=  \prod\n_{n\in A} (\bar z_n/2) $.
Moreover $F_z\ge 0$ and $\int F_z dm=1.$\\
Let $f_z=\sum\n_{n\in \Lambda} z_n x_n e^{int}$ and $f= \sum\n_{n\in \Lambda}   x_n e^{int}$.
Then $\Re(f_z)\ast F_z= \Re(f)/2. $  
Therefore by the convexity of the exponential function 
$$\int e^{\Re(f)/2} dm \le \int(\int F_z(s) e^{\Re(f_z(t-s))} dm(s)) dm(t)$$
and by Fubini and the translation invariance of $m$
this implies
$$\int e^{\Re(f)/2} dm \le \int F_z(s) dm(s) \int e^{\Re(f_z(t))} dm(t)= \int e^{\Re(f_z(t))} dm(t).$$
We now average the right hand side over $z$
with respect to the normalized Haar measure 
on the group $G= \T^\Lambda$. By Fubini this gives us
$$\int e^{\Re(f)/2} dm \le   \int e^{\Re(f_z(t))} dm(t)dm_G(z)=\int(\int e^{\Re(f_z(t))} dm_G(z)  )dm(t)$$
and since we already know that $sg(\{z_n\})\le 1$
(or equivalently $sg(\{\bar z_n\})\le 1$)
we find
$$\int e^{\Re(f)/2} dm \le \exp( \sum|x_n|^2/2) ,$$
and we conclude by homogeneity that $sg(\{ e^{int}\mid {n\in \Lambda}\})\le 2$.\\
It is easy to check that a Hadamard lacunary sequence is a finite union of sequences $(k(n))$ satisfying $ k(n+1)/k(n)\ge 2$ for all $n$. Since such sequences
are clearly quasi-independent (see Remark \ref{r10})
the second assertion follows.
\end{proof}

More generally, let us replace $\T$ by
a compact Abelian group $G$ equipped with its
normalized  Haar measure $m_G$. The dual group $\hat G$
is the discrete group formed of all the continuous characters
on $G$. A character is   a homomorphism $\gamma:\ G\to \T$.
The group operation on $\hat G$ is the pointwise product of characters.
When $G=\T$ the characters are all of the form
$\gamma_n(z)= z^n$  ($z\in \T$) for some $n\in \Z$.
The correspondence $\gamma_n\leftrightarrow n$
allows us to identify $\hat \T$ with $\Z$ as discrete groups
(pointwise multiplication on $\hat \T$ corresponds to addition on $\Z$).

\begin{rem}\label{ru2} The implication quasi-independent 
$\Rightarrow$ subgaussian remains clearly valid with the same proof
for a subset $\Lambda$ of any discrete group $\hat G$.
\end{rem}
\begin{thm}\label{th1} Let $(f_1,\cdots f_n)$ be subgaussian  
on a probability space $(T,m)$ with $sg(\{f_k\})\le s$.
Assume that $\|f_k\|_2=1$ and $\|f_k\|_\infty \le C$.
Then for any $0<\d<1/C$ there is a subset $\cl T\subset T$ with
$$\log |\cl T| \ge   n  (1 - \d  C)^2/( 2s^2 C^2),$$
such that for any $x\not = y\in \cl T$ we have 
$$(\sum\n_1^n |f_k(x) -f_k(y)|^2 )^{1/2} > \d \sqrt n.$$
\end{thm}

\begin{proof} Let $\cl T$ be a maximal subset with this property.
Then for any $x\in T$ there is $y\in \cl T$ such that
$(\sum\n_1^n |f_k(x) -f_k(y)|^2 )^{1/2} \le  \d \sqrt n$, and hence
$$\sum\n_1^n |f_k(x)  |^2=
\sum\n_1^n \Re(f_k(x)\ovl{f_k(x)})
\le \sum\n_1^n \Re(f_k(x)\ovl{f_k(y)})+ \d n C.$$
Therefore
for any $\lambda>0$
\begin{align*}\exp{\lambda n}&= \exp{\lambda \sum\n_1^n \int |f_k|^2dm }
\le \int \exp{\lambda \sum\n_1^n |f_k(x)  |^2 } m(dx) \\
&\le
e^{\lambda\d n C} \int \exp{(\lambda\sup\n_{y\in \cl T} \sum\n_1^n \Re(f_k(x)\ovl{f_k(y)}) )} m(dx) \\&\le e^{\lambda\d n C} \sum\n_{y\in \cl T} \int \exp{(\lambda  \sum\n_1^n \Re(f_k(x)\ovl{f_k(y)}) )} 
 m(dx) \\ &\le e^{\lambda\d n C} |\cl T| \exp{(\lambda^2 s^2 nC^2/2)}.
 \end{align*}
Therefore
$$|\cl T|  \ge \exp{n(\lambda    (1 - \d  C) -\lambda^2 s^2 C^2/2) }. $$
Choosing $\lambda= ( s^2 C^2)^{-1} (1 - \d  C)$ (to maximize the last expression) 
we obtain the announced inequality.
\end{proof}
\begin{rem} Note that in the preceding proof 
instead of $sg(\{f_k\})\le s$ it suffices to assume
 $sg(\sum \n_1^n x_k f_k)\le s \sqrt n \sup |x_k|$
 for any $x_k\in \C$.
\end{rem}
In   Theorem \ref{th1}, we have obviously
$|T|\ge |\cl T|$. In particular:
\begin{cor}\label{c1} Let $(f_1,\cdots f_n)$ be subgaussian  characters
on a finite Abelian group $G$ with $sg(\{f_k\})\le s$.
Then  for any $0<\d<1$ $$\log |G|\ge (1-\d)^2/2s^2   .$$
\end{cor}
\begin{cor} 
In the situation of Theorem \ref{th1}, assume in addition
that  $(f_1,\cdots f_n)$ are 
continuously differentiable functions on $([0,2\pi],dt/2\pi)$.
Then 
$$n^{-1/2}\|(\sum |f_k'|^2)^{1/2}\|_\infty \ge \frac{\d  }{2\pi} \left( \exp{(  n  (1 - \d  C)^2/( 2s^2 C^2)  )}-1\right) .$$
\end{cor}
\begin{proof}  Let $L=n^{-1/2}\|(\sum\n_1^n |f'_k|^2)^{1/2}\|_\infty$
We have for any $x,y\in [0,2\pi]$
$$(\sum\n_1^n |f_k(x) -f_k(y)|^2 )^{1/2} \le  n^{1/2} L |x-y|.$$
Therefore for any $x\not = y\in \cl T$
$$|x-y|\ge \d/L.$$
But obviously, we cannot find more that $1+2\pi L/\d$ points
in   $[0,2\pi]$ with mutual distance $\ge \d/L$.
Thus we conclude $2\pi L/\d\ge |\cl T|-1$
\end{proof}
\begin{cor}\label{cca} If $\Lambda\subset [1,\cdots,N]$ (or if $\Lambda$
is included in an arithmetic progression of length $N$)
and 
  $sg(\{ e^{int} \mid n\in\Lambda \})\le s$,
then
$$\log(\frac{2\pi N}{\d} +1)\ge |\Lambda| (1-\d)^2/2s^2 .$$
\end{cor}
\begin{proof} 
The case of an arithmetic progression of length $N$
can be reduced to $[1,\cdots,N]$.
For $f_k= e^{i k(n)t}$ with $1\le k(n)\le N$ we have $L\le N$.
\end{proof}
\begin{rem} If  $\Lambda=\{2^k\mid 1\le 2^k \le N\}$
then $\log N\approx |\Lambda|$; so the logarithmic
growth rate
for the intersection of  a subgaussian set
with any arithmetic progression of length $N$ given by  Corollary \ref{cca} is essentially optimal.
\end{rem}
\begin{rem}\label{ru} Let $\Lambda_1=\{f_n\}$ and 
$\Lambda_2=\{h_n\}$ be two subgaussian families of functions 
on the same probability space. Then the union
 $\Lambda_1 \cup \Lambda_2$ is subgaussian.
 This follows from \eqref{69}. 
 \end{rem}
 
\section{Subgaussian sets of integers, arithmetic characterization}

We will now describe the existing arithmetic characterization
of subgaussian sets of integers and, in the next section,  the main open problem
concerning them.\\
For any finite set $\Lambda\subset \Z$  or more generally
$\Lambda\subset \hat G$ (here  $\hat G$ is any discrete Abelian group
denoted additively), let
$$R(\Lambda)=\{ \xi\in \{-1,0,1\}^\Lambda \mid \sum\n_{n\in \Lambda} \xi_n n=0\}.$$ 
In other words $R(\Lambda)$ is the set of relations with coefficients
in $\{-1,0,1\}$ satisfied by $\Lambda$.
Note that $\Lambda$ is quasi-independent iff $| R(\Lambda) | =1$.
The number  $R(\Lambda)$ is related to Fourier series
by the following obvious identity, valid for any finite subset $A\subset \Lambda$
 \begin{equation}\label{101}R(A)=\int \prod\n_{n\in A} (1+ e^{int}+ e^{-int}) dm(t) .\end{equation} 
 The number $N(k,m,n)$ introduced in the next statement
appears in the theory    of constant weight codes,  
see Remark \ref{code} below.
\begin{lem}\label{gv} Let $k<m<n$ be integers. As usual
let $[n]=\{1,\cdots,n\}$. Let $N(k,m,n)\ge 1$ be the largest
possible cardinal of a family
$\cl T$ of subsets of $[n]$ such that
 \begin{equation}\label{98}\forall   t\in \cl T\quad |t|=m \text{ and } \forall s\not= t\in \cl T\quad |s\cap t|\le k.\end{equation}
Let $A\subset \Z$ be a subset with  $|A|=n$.
If  $R(A)< N(k,m,n)$, then  $A$ contains a quasi-independent subset $B\subset A$ with
$$|B|\ge m-k.$$
\end{lem}
\begin{proof} 
Since $A$ and $[n]$ are in bijection, we may assume
that $\cl T$ is a family
 of subsets of $A$.
For any $t\in \cl T$ consider a maximal subset $r_t\subset t$
that supports a relation, i.e. for $r=r_t$  there exists   $ (\xi_n)\in \{-1,1\}^r$
such that $\sum \xi_n n=0$ and 
there is no larger subset of $t$ satisfying this. We claim
that for some $t$ we must have $|r_t|\le k$.
Otherwise, $|r_t|> k$ for all $t$. But since $|s\cap t|\le k$
for all $s\not= t$, the mapping $t\mapsto r_t$ must be one to one.
To each $r_t$  we can associate (by adding several zeros) a relation
$\xi^t \in  \{-1,0,1\}^A$ such that $\sum \xi^t_n n=0$ 
with support $r_t$. Obviously $t\mapsto \xi^t$ is also one to one.
Thus we obtain $|\cl T|\le |R(A)|$, contradicting our assumption
that  $R(A)< N(k,m,n)$. This proves our claim.
Now choose $t$ so that $|r_t|\le k$.
Let $B=t\setminus r_t$. We have $|B|\ge m-k$
and the maximality of $r_t\subset t$ implies that there cannot be any
nontrivial
relation with coefficients $\pm 1$ supported inside $B$. In other words
$B$ is quasi-independent.
\end{proof}

\begin{lem}\label{gv2}
Assuming that $3n/8, n/2$ are integers, we have
 \begin{equation}\label{990} N(3n/8, n/2, n) \ge c'  \exp{(n/17)},\end{equation}
 where $c'>0$ is    independent of $n$.
\end{lem}
\begin{proof} Let $Q$ be the uniform probability over 
all the $2^n$ subsets of $[n]$.
Let $\cl T$ be a maximal family of subsets satisfying \eqref{98}.
Then for any $A\subset [n]$ with $|A|=m$ and $A\not\in \cl T$ 
there is $t\in \cl T$ such that
$|A\cap t|>k$ (otherwise we could add $A$ to $\cl T$ contradicting its maximality). 
Actually, if  $A\in \cl T$,  then $t=A$ trivially satisfies $|A\cap t|=m>k$.
Therefore $\{A\mid |A|=m\}\subset \cup_{t\in \cl T} \{A\mid |A|=m, |A\cap t|>k\}$ and hence
 \begin{equation}\label{99} Q(\{A\mid |A|=m\}) \le |\cl T|  \sup\n_{t\in \cl T} Q(\{ A\mid |A|=m, |A\cap t|>k\}).\end{equation}
Of course, whenever $|t|=m$, the numbers $Q(\{ A\mid |A|=m, |A\cap t|>k\})$
are all the same and hence
$Q(\{ A\mid |A|=m, |A\cap t|>k\})=Q(\{ A\mid |A|=m, |A\cap [m]|>k\})$. By an easy counting
argument,
 the cardinal of $\{ A\mid |A|=m, |A\cap [m]|>k\}$ is equal
 to $\sum\n_{k<j\le m} {m  \choose j}  {n-m  \choose m-j}$.
 Although we could use combinatorics, we prefer to use probability
 to estimate this number. 
 Let $(\vp_j)$ be in $\{-1,1\}^n$ and let $P$ be the uniform probability
 on $\{-1,1\}^n$. We have a $1-1$ 
 equivalence 
 between $P$ and $Q$ using the correspondence
 $\vp=(\vp_j) \mapsto A=\{j\mid \vp_j=1\}$. Note $|A|=  \sum (\vp_j  +1)/2$.
 Let $S_n= \sum\n_1^n \vp_j$ so that  $|A|=(S_n+n)/2$
 and $ |A\cap [m]|=(S_m+m)/2$.
 Thus \eqref{99} implies
 \begin{equation}\label{99'} P(\{ S_n=2m-n\})\le   |\cl T|   P(\{ (S_m+m)/2>k\})=  |\cl T|  P(\{ S_m>2k-m\}).
 \end{equation}
 By a well known bound
 there is a positive number $c_0>0$ (in fact $c_0=1/\sqrt 2$)
 so that assuming $n$ even ${n \choose n/2}\ge c_02^n/\sqrt n$.
 Thus  assuming $n=2m$ and $2k-m=m/2 $ we find by \eqref{99'}
 (using  \eqref{vp} and \eqref{102})
  $$c_0/\sqrt n \le   |\cl T|   \exp{-((2k-m)^2/2m)}\le  |\cl T|   \exp{-m/8}\le  |\cl T|   \exp{-n/16},$$
 and we obtain $N(3n/8, n/2, n) \ge (c_0/\sqrt n)  \exp{(n/16)}$, from which \eqref{990}  follows a fortiori.
\end{proof}
\begin{rem}\label{code}
The number $N(k,m,n)$ introduced in Lemma \ref{gv}
appears in the theory    of constant weight codes
where it is denoted by $A(n, 2(m-k) ,m)$.
A code word is a sequence of $0$'s and $1$'s,  its length   is the number
of $0$'s and $1$'s, and its weight is the number of $1$'s.
The  Hamming distance between any two such words
is the number of places where they differ.
 Thus $N(k,m,n)$ is equal to the maximal number of 
 code words of length $n$ with weight $m$
 and mutual Hamming distance at least $2(m-k)$.
 The simple packing argument used for Lemma \ref{gv2} is a variant
 of a famous estimate known in Coding Theory as the Gilbert-Varshamov bound,
 adapted to the weight $m$ case. It is known (this seems to be in the coding folklore)
 that if $m=[a n]$, $k=[bn]$ with $0<b<a^2$ and $0<a\le 1/2$
 then  (assuming $n\to \infty$)  we have an exponential lower bound $N(k,m,n)\ge 2^{\d n}$ for some $\d=\d(a,b)>0$. The proof of Lemma \ref{gv2} can be modified to yield that.
 It seems however that no sharp formula is known for $\d=\d(a,b)>0$.
 See
\cite[chap. 17, \S 2]{MS} for more   on this vast subject.
I am grateful to Noga Alon for  the information and  references used in the present remark.
\end{rem}
\begin{lem}\label{l105}  Assume again that $3n/8, n/2$ are integers.
 If $|A|=n$ and $|R(A)|< c'  \exp{(n/17)}$
 then $\exists B\subset A$ quasi-independent with $|B|\ge n/8$.
\end{lem}
\begin{proof} This is  immediate from the preceding two Lemmas.
\end{proof}
\begin{thm}\label{erd} Let $  \Lambda\subset \Z\setminus\{0\}$  or more generally
$ \Lambda\subset \hat G \setminus\{0\}$ ($\hat G$   any discrete Abelian group).
The following are equivalent:
\begin{itemize}
\item[(i)]  $\Lambda$ is subgaussian. 
\item[(ii)]  There is $\d>0$ such that   any
 finite   $A\subset \Lambda$ contains a quasi-independent subset
  $B\subset A$ with $|B|\ge \d |A|$.
 \item[(iii)] There is $\d>0$ and $s>0$ such that   any
 finite   $A\subset \Lambda$ contains a (subgaussian) subset
  $B\subset A$ with $|B|\ge \d |A|$ and  $sg(B)\le s$.
   \item[(iii)'] For any $0<\d<1$ there is $s>0$ such that   any
 finite   $A\subset \Lambda$ contains a (subgaussian) subset
  $B\subset A$ with $|B|\ge \d |A|$ and  $sg(B)\le s$.
   \item[(iv)] There is a constant $C$ such that
   for any finite subset $A\subset\Lambda$ we have
   $$sg(\Re( \sum\n_{n\in A} e^{int}))\le C|A|^{1/2}.$$
\end{itemize}
\end{thm}
\begin{proof} 
Assume (i).
Then there is $C$ such that for any finite subset
$A\subset \Lambda$ with $|A|=n$ the function
$S_A(t) =\Re( \sum\n_{k\in A} e^{ikt})=\sum\n_{k\in A}  \cos(kt) $ is subgaussian with
$sg(S_A)\le C|A|^{1/2}$. Then for any $0<\d<1$ we have
$$\int \prod\n_{{k}\in A}  (1+ \d \cos({k}t)) dm(t) \le \int   \exp{(\d S_A)} dm(t)\le \exp {(C\d^2 n/2)}.$$
Let $(\d_{k})_{{k}\in A}$ be an i.i.d. family of $\{0,1\}$-valued
variables with $\P(\{\d_{k}=1\})=\d/2$.
Let $A(\omega)=\{{k}\mid \d_{k}(\omega)=1\}$.
Then $\prod\n_{{k}\in A(\omega)}  (1+ e^{i{k}t}+ e^{-i{k}t})=\prod\n_{{k}\in A }  (1+\d_{k}(\omega)( e^{i{k}t}+ e^{-i{k}t}))$ and hence
$$\E \int \prod\n_{{k}\in A(\omega)}  (1+ e^{i{k}t}+ e^{-i{k}t}) dm(t)=\int \prod\n_{{k}\in A}  (1+ \d  \cos({k}t))\le \exp {(C\d^2 n/2)}.$$
In other words
 \begin{equation}\label{104}\E |R(A(\omega))| \le \exp {(C\d^2 n/2)}.\end{equation}
But we also have 
$|A(\omega)| -\d n/2 = \sum\n_{{k}\in A } ( \d_{k} -\E\d_{k} )$, and hence by well known bounds
for a sum of independent mean $0$ variables with values in $[-1,1]$
(indeed a very particular case of Theorem \ref{chp6athmH5}  
with $d_n= \d_n -\E\d_n   $ tells us that $sg(\sum\n_{{k}\in A } ( \d_{k} -\E\d_{k} )) \le n^{1/2}$ then we may use  \eqref{102})
$$\forall c>0\quad \P(\{ |A(\omega)| -\d n/2  < -c \}) =\P(\{ \sum\n_{n\in A }  (\d_n -\E\d_n) < -c \}) \le \exp {(-c^2/2n)}.$$
Therefore
$\P(\{ |A(\omega)| -\d n/2   <- \d n/4  \}) \le \exp {(-\d^2 n/32)}$, and hence
$$\P(\{ |A(\omega)| \ge \d n/4\}) \ge 1- \exp {(-\d^2 n/32)}.$$
By \eqref{104}
$$\P(\{ |R(A(\omega))| \le 2\exp {(C\d^2 n/2)} \}) \ge 1/2.$$
Assume  \begin{equation}\label{nob}
1/2 + 1- \exp {(-\d^2 n/32)} >1.\end{equation}
Then for some $\omega$ we have both $ |A(\omega)| \ge \d n/4$
and $|R(A(\omega))| \le 2\exp {(C\d^2 n/2)}$, and hence
$$|R(A(\omega))| \le 2\exp {(2C\d|A(\omega) | )}.$$
We now choose $\d=\d_C$ so that $2C \d_C=1/18<1/17$.
Then $|R(A(\omega))| \le 2 \exp {(|A(\omega) |/18 )}.$
Note $ |A(\omega)| \ge \d_C n/4$. Therefore there is
clearly a large enough number $N$ (depending only on $C$)
such that for all $n\ge N$ both \eqref{nob}   and  (for the $\omega$ we 
select)
$ 2 \exp {(|A(\omega) |/18 )} < c' \exp {(|A(\omega) |/17 )}$ hold,
and hence 
$$|R(A(\omega))|  < c' \exp {(|A(\omega) |/17 )}.$$
By Lemma \ref{l105} this implies that $A(\omega)$ contains
a quasi-independent subset $B$ with $|B|\ge |A(\omega)|/8\ge \d_C n/32$.
 (We ignore
the requirement that $3|A(\omega)|/8, |A(\omega)|/2$ be integers, which is easy to bypass by replacing $A(\omega)$ by a maximal subset with cardinal dividable by 8.)
This proves (ii) since the sets with $n\le N$ are easily
treated by adjusting the number $\d$ appearing in (ii) small enough.\\
(ii) $\Rightarrow$ (iii) follows from Proposition \ref{pqi}.\\
 Assume  (iii). Let $|A|=n$. Let $B\subset A$ be given by (iii), i.e.
$sg(B)\le s $  and $|B|\ge \d n$. We may apply (iii) again to $A\setminus B$.
This gives us $B_1\subset A\setminus B$   with $sg(B_1)\le s $ and
$|B_1|\ge \d |A\setminus B|$. Now let $B'= B\cup B_1$.
We have $|B'|\ge (\d + \d(1-\d) )n $
and, by Remark \ref{ru} and \eqref{69},   also $sg(B')\le 2s$. Thus we have improved $\d$ from 
the value $\d$ to $\d_1=\d + \d(1-\d)$. Iterating this process, we easily
obtain (iii)'
\\
 Assume (iii)'. Let $C(n) $ be the smallest constant $C$ such that
 $sg(S_A)\le C\sqrt{|A|}$  for all subsets $A\subset\Lambda$ with $\le n$ elements.
Let $|A|\le n$.
We fix $0<\d<1$ suitably close to $1$ (to be determined).
 Let $B\subset A$ be given by (iii)', so that
$sg(S_B)\le s\sqrt{|B|} $ and $|B|\ge \d |A|$. 
We have obviously by definition of $C(n)$
$sg(S_{A\setminus B}) \le C(n)  \sqrt {n(1-\d)}$.
By \eqref{69}, $$sg(S_{A})^2\le  2( s^2{|B|} +  C(n) ^2 {|A|(1-\d)} )
\le 2 s^2 |A| + 2 (1-\d)  C(n)^2 |A|,$$
which implies
$$C(n)^2 \le 2 s^2  + 2 (1-\d)  C(n)^2  .$$
Thus if $\d$ is chosen so that $\d_1=2 (1-\d)<1$
we conclude
$$C(n)^2 \le  (1-\d_1)^{-1} 2 s^2,$$
which shows that $C(n)$ is bounded, so that (iv) holds.

The proof that (iv) $\Rightarrow$ (i) is more delicate. We skip the details.
This was first proved in \cite{Pi2} using the Dudley-Fernique
metric entropy condition together with a certain interpolation argument.
Bourgain \cite{Bo} gave a completely different proof.
Both proofs show that (iv) implies that $\Lambda$ is Sidon, as defined
below,
and then Sidon implies subgaussian (see Theorem \ref{rud}).
\end{proof}
\begin{rem} Note that in the proof that (i) $\Rightarrow$ (ii)
we actually showed that (iv) $\Rightarrow$ (ii). Thus we gave a
complete proof of the equivalence of (ii), (iii), (iii)' and (iv).
\end{rem}
\begin{rem}\label{ia}  The proof that (iv) $\Rightarrow$ (i) in \cite{Pi2}
passes through the following
\begin{itemize}
  \item[(v)]  Let $1<p<2$. There is a constant $C$ such that
  for any $f$ in the linear span of $\Lambda$ we have
  $$\|f\|_{\psi_{p'}} \le C (\sum\n_{n\in \Lambda} |\hat f(n)|^p)^{1/p}.$$
\end{itemize}
We show in \cite{Pi2} that (iv) $\Rightarrow$ (v) (this is an  argument from the so-called real interpolation method).
Then using special properties  of the metric entropy integrals
we show that (v) $\Rightarrow$ Sidon, and  hence (v) $\Rightarrow$ (i) follows by Theorem \ref{rud}.\end{rem}

\section{Main open problem}\label{mop}

We now come to the main open problem concerning subgaussian sets
(or equivalently Sidon sets, that are defined in the next section) of characters on a compact Abelian group $G$.

\noindent{\bf Conjecture.}\emph{ Any subgaussian set
  is a finite union of quasi-independent sets.}

The conjecture is supported by the case when $G=\Z(p)^\N$.
Here $p>1$ is a prime number and $Z(p)=\Z/p\Z$ is the field
with $p$ elements.
We have $\hat{Z(p)}=Z(p)$. Indeed, any  $n\in \Z/p\Z$
(represented, if we wish, by  a number $n\in [0,p-1]$ modulo $p$)
defines a character $\gamma_n$ on $Z(p)$ by
$$\forall t\in Z(p)\quad \gamma_n( t)= e^{2\pi i tn/p}.$$
As before for $\Z$, the correspondence $n\leftrightarrow \gamma_n$
allows us to identify $\hat{Z(p)}$ with $Z(p)$.
One can also associate to $\gamma_n$
the $p$-th root of unity $\gamma_n(1)=e^{2\pi i n/p}.$

Let $ \Z(p)^{(\N)} \subset \Z(p)^{\N}$ denote the set of sequences $n=(n_k)\in \Z(p)^{\N}$
with only finitely many nonzero 
terms.
Let $n=(n_k)\in \Z(p)^{(\N)}$. Then the function $\gamma_n:\ \Z(p)^{\N} \to \T$
defined by
 $$\forall t=(t_k)\in   Z(p)^\N\quad\gamma_n( t)= e^{2\pi i \sum t_k n_k /p},$$
is a character on $Z(p)^\N$, and all the characters are of this form.
Thus again $n\leftrightarrow \gamma_n$
allows us to identify $\hat{Z(p)^\N}$ with $Z(p)^{(\N)}$.

The novel feature is that the group $\hat G=\Z(p)^{(\N)}$ is a vector space
over the field $Z(p)$. Of course
the scalar multiplication by $m\in \Z(p)$  is defined on $Z(p)^{(\N)}$
in the natural way 
$$\forall n=(n_k)\in Z(p)^{(\N)}\quad m\cdot n= (mn_k).$$
For $G= {Z(p)^\N}$, a complete description 
of subgaussian sets of characters on $G$
was given by Malliavin and Malliavin
\cite{MM}.
\begin{thm}[\cite{MM}]\label{th2}
Let  $p>1$ be a prime number.
 Let $G= {Z(p)^\N}$ and $\hat G=\Z(p)^{(\N)}$.
Let $\Lambda\subset \hat G\setminus\{0\}$.
The following are equivalent:
\begin{itemize}
\item[(i)]  $\Lambda$ is subgaussian. 
 \item[(ii)]  $\Lambda$ is a finite union of 
 linearly independent sets over the field $\Z/p\Z$.
\item[(iii)]  $\Lambda$ is a finite union of quasi-independent sets.
\end{itemize}
\end{thm}
The miracle that produces this beautiful result
is a deep (and difficult) combinatorial fact in linear algebra
due to Horn \cite{Ho} (published also by Rado but 10 years later),
that says the following:
\begin{thm}[\cite{Ho}]\label{ho}
Let $\Lambda$ be a   subset of a vector space over any field. Let $k>0$ be an integer.
Assume that  any finite subset $A\subset \Lambda$
contains a (linearly) independent  subset $B\subset A$ with
$|B|\ge |A|/k$.
Then (and only then) $\Lambda$ can be decomposed as a union of
$k$ (linearly) independent  subsets.
\end{thm}
Note that the assumption is clearly necessary for the conclusion to hold.  
\begin{proof}[Proof of Theorem \ref{th2}] 
Assume (i). We will apply the criterion of Theorem \ref{ho}.
Let $A\subset \Lambda$ be a  finite subset.
Let $B$ be a maximal   independent  subset of $A$ over the field $\Z/p\Z$.
Then $A$ must be included in the vector space $V_B$ generated by $B$
(indeed, if not we would find an element that we could add to $B$,
and that would contradict the maximality of $B$).
Clearly $\dim(V_B)=|B|$ and hence $|V_B|=p^{|B|}$. 
But now a fortiori $V_B$ is finite group, and $sg(A)\le sg(\Lambda)$,
therefore by Corollary \ref{c1}
we have
 for any $0<\d<1$ $$\log |V_B|\ge |A|  (1-\d)^2/2sg(\Lambda)^2  ,$$
 and hence if $\kappa= (1-\d)^{-2} 2sg(\Lambda)^2 \log(p)$
 and if $k$ is the smallest integer such that $k\ge \kappa$
 $$|B|\ge |A|/\kappa\ge |A|/k.$$
By Theorem \ref{ho} (ii) follows. Then
(ii) $\Rightarrow$ (iii) is obvious
and (iii) $\Rightarrow$ (i) follows from Remarks \ref{ru} and \ref{ru2}.\end{proof}
\begin{rem} In \cite{Bop} Bourgain generalized (i) $\Leftrightarrow$ (iii) in Theorem
\ref{th2}
to the case  when $p=\prod p_k$  where $p_1,\cdots,p_n$ are distinct prime numbers. However, it seems that  (i) $\Leftrightarrow$ (iii)
is still an open problem even for $p=4$.
\end{rem}

\begin{rem} Let $(\gamma_n)$  ($n\in \N$) be any sequence of characters on a compact Abelian group $G$. Thus each $\gamma_n$
can be viewed as a random variable on $(G,m_G)$ with values in $\T=\{z\in \C\mid |z|=1\}$.
  Assume first that there is no ``torsion", i.e.
that $\gamma^\xi_n\not=1$ for any $\xi\not =0$ ($\xi\in \Z$).
Then $(\gamma_n)$  are stochastically independent as random variables
iff for any sequence $(\xi_n)\in \Z^{(\N)}$ not identically  $=0$
$$\prod  \gamma^{\xi_n}_n \not\equiv1.$$
Equivalently, for any such $(\xi_n)$
$$\int \prod  \gamma^{\xi_n}_n dm_G=0.$$
Indeed, this condition 
holds   iff for any $n$ and any
polynomials $Q_n(z,\bar z)$ we have for any $n$
  $$\int \prod\n_1^n  Q_k(\gamma _k) dm_G= \prod\n_1^n \int  Q_k(\gamma _k) dm_G.$$
  To check this just replace polynomials by monomials.
 
 Now assume (``torsion group") that there is a positive integer $p_n$ such that
 $\gamma_n^{p_n}=1$. We choose $p_n$ minimal and we assume $p_n>1$.
 Note that $\gamma_n^\xi=1$ iff $\xi\in p_n\Z$.
 Then $(\gamma_n)$  are stochastically independent as random variables
iff for any sequence $(\xi_n)\in [0,p_n-1]^{(\N)}$ not identically  $=0$
$$\prod  \gamma^{\xi_n}_n \not\equiv1.$$

This shows that quasi-independence appears as a weaker form of 
stochastic independence.  
However, if $G=\{-1,1\}^\N$ and if
$\gamma_n$ is the $n$-th coordinate on $G$ then 
the two forms of independence coincide (here $p_n=2$). This corresponds
to the usual random choices of signs, as in Remark \ref{vp}.

We note in passing that the classical Rademacher functions $(r_n)$, which
are defined on $([0,1],dt)$ by
$$\forall n\ge 0\quad r_n(t)={\rm sign}(\sin(2^n (2\pi t)))$$
form an i.i.d. sequence of uniformly distributed choices of signs.
In sharp contrast, the sequence $(\exp{i2^n (2\pi t)})$
is only quasi-independent as a sequence of characters on $\T$.
\end{rem}
\begin{rem}\label{pal} Any Hadamard lacunary sequence $\Lambda=\{n(k)\}$ is a finite union of quasi-independent
sets. Indeed, if 
\eqref{h01} holds there must exist a number $N$
such that  $$\forall n\quad|\Lambda \cap (2^n,2^{n+1}]|\le N.$$
This implies that $\Lambda$ is the union of $N$ sequences
satisfying $|\Lambda \cap (2^n,2^{n+1}]|\le 1$ for all $n$.
But then (by separating the $n$'s into evens and odds)
 each such sequence is the union of two sequences
 such that $k(n)>\sum\n_{j<n} k(j)$ which by 
Remark \ref{r10} are quasi-independent.
\end{rem}
\begin{rem} There are quasi-independent sets in $\N$ that are not finite unions
of Hadamard lacunary sets. Indeed, if $\Lambda$ is such a finite union,
then it is easy to see that there is a number $K$ such that
$|\Lambda \cap [2^n,2^{n+1})|\le K$ for any $n\ge 1$. 
The set $\{ 4^{n^2}+ 2^j\mid n\ge 1, 1\le j\le n\}$ clearly violates that,
but it is an easy exercise to check that it is quasi-independent.
\end{rem}
\begin{rem}[``Condition de maille"]  By a variant of the argument
in Theorem \ref{th1}, one can show that any subgaussian set $\Lambda
\subset \Z$ satisfies the following condition: there is a constant $  K>0$ such that for any $n,s>0$ and any
$k_1,\cdots,k_n\in \Z$ 
$$|\Lambda \cap \{k_1m_1+\cdots+k_n m_n\mid |m_1|+\cdots+|m_1|\le 2^s\}|\le Kns.$$
See \cite[p. 71]{Ka2} for details.
It seems to be still open whether this characterizes subgaussian sets. 
\end{rem}
By Theorem \ref{erd}, the conjecture highlighted in this section is equivalent
to the following purely combinatorial\\
 \noindent{\bf Problem:} {\it Let $\Lambda\subset \Z$.
Assume that there is $\d>0$ such that any finite subset
$A\subset \Lambda$ contains a quasi-independent   $B\subset A$
with $|B|\ge \d |A|$, does it follow that $\Lambda$ is a finite union of quasi-independent sets ?}

In 1983, I drew   Paul Erd\"os's  attention
to this problem, raised in \cite{Pis}.
He became interested
in the classes of sets that one could substitute
to that of quasi-independent  sets for which 
the problem would have an affirmative answer (see \cite{ENR,ENR2}).
He and his co-authors considered  generalizations of the problem
for graphs or hypergraphs, but the problem
remains open.

\section{Subgaussian bounded mean oscillation}\label{bmo}

The goal of this section is to show that the sequences of positive
integers that can be written as a finite union of Hadamard-lacunary ones
can be characterized as those that are subgaussian and remain subgaussian
uniformly when restricted to an arbitrary subarc $I$ equipped with its
normalized Lebesgue measure $m_I$.

Here we prefer to think of $\T$ as the unit circle in $\C$. By a subarc we mean
 a connected subset of $\T$
with non empty interior.
We denote by $\cl I$ the collection of all subarcs
in $\T$. For any $I\in \cl I$,
let $m_I=1_I dt/|I|$ (normalized Lebesgue measure on $I$).
For any $f\in L_1(\T)$ we set  
$$f_I=\int_I f dm_I \quad\text{and} \quad\|f\|_{*,1}=\left|\int fdm\right|+\sup_{I\in \cl I} \|f-f_I\|_{L_{1}(dm_I)}\in [0,\infty].$$
Note that for a complex-valued $f\in L_1(\T)$ its real and imaginary parts
satisfy obviously
$$\| \Re(f)\|_{*,1} \le \|f\|_{*,1}\quad\text{and} \quad  \| \Im(f)\|_{*,1}   \le\|f\|_{*,1}
.$$
The space BMO($\R$) (resp.  BMO($\C$))
 of functions with bounded mean oscillation is defined
as formed of all those \emph{real-valued} (resp.  \emph{complex-valued}) $f\in L_1(\T)$ such that $\|f\|_{*,1}<\infty$.
Equipped with the norm $f\mapsto \|f\|_{*,1}$ it becomes a  real (resp. complex) Banach space.
This space is of crucial importance in the theory of $H^p$-spaces (see e.g. \cite{Gar}).\\
The main point is that
  BMO($\R$)  is      the dual of $H^1$ (Fefferman's theorem).
  A priori, the space $H^1$ is a complex Banach space but  for this duality theorem
  we view it as real space.
Here we define $H^1$ as the closure in $L_1(\T)$ of the linear span,
denoted by $\cl P_+$, 
of the functions $\{e^{int}\mid n\ge 0\}$.  We equip it 
with the norm induced by $L_1$, that we denote by $\| \ \|_{H^1}$.
  Fefferman's  inequality  establishes the duality, as follows:
  \begin{equation}\label{feff}\exists C_F>0\ \forall f\in {\rm BMO}(\R),\ \forall x\in \cl P_+\quad \left|\int f \Re(x)dm\right|\le C_F \|f\|_{*,1}\|x\|_{H^1}.\end{equation}
This shows that we can associate to each $f  \in {\rm BMO}(\R)$
an  $\R$-linear form $\xi_f:\ H^1\to \R $, obtained by densely extending
the functional
$x\mapsto \xi_f(x)=\int f \Re(x)dm$ from $\cl P_+$ to the whole of $H^1$.
It turns out that any $\R$-linear form $\xi:\ H^1\to \R $ is of this form.
Moreover, the norm $\|f\|_{*,1}$ is equivalent
to the norm of $\xi_f:\ H^1\to \R$.   
In other words, BMO$(\R)$ can be identified with
the space of bounded $\R$-linear
forms on $H^1$. We call the latter space the  $\R$-linear dual of $H^1$, it is the dual of $H^1$ when we view the latter as a real Banach space.
Thus the content of Fefferman's duality theorem is that
    BMO$(\R)$  is the  $\R$-linear dual of $H^1$.
We refer the reader to \cite{Gar} for more on these topics.

In order to discuss other equivalent norms
on the space BMO, for any $a>0$ and $f\in L_1(\T)$ we define   
 $$\|f\|_{*,\psi_a}=\left|\int fdm\right|+\sup_{I\in \cl I} 
 \|f-f_I\|_{L_{\psi_a}(dm_I)}\in [0,\infty].$$
A famous theorem of John and Nirenberg asserts that
$f\in BMO(\C)$ iff $\|f\|_{*,\psi_1}<\infty$ and the norms
$f\mapsto\|f\|_{*, 1}$  and $f\mapsto\|f\|_{*,\psi_1}$ are equivalent.
(A fortiori, the same holds for $f\mapsto\|f\|_{*,\psi_a}$ for any $0<a\le1$.)
This particular fact is even valid for Banach space valued functions.
We refer to our recent book \cite{Pi4} for more information on  Banach
space valued $H^p$-spaces.

It is well known that the norms $f\mapsto\|f\|_{*, 1}$ or $f\mapsto\|f\|_{*,\psi_1}$ are \emph{not}
equivalent to the norm $f\mapsto\|f\|_{*,\psi_a}$ when $a>1$.
Nevertheless, the norm $f\mapsto\|f\|_{*,\psi_2}$
is equivalent to the usual BMO norm $f\mapsto\|f\|_{*, 1}$
when restricted to $f$ in the linear span of $\{\exp{(in_kt)}\}$
if the sequence $\{n_k\}$
is a finite union of Hadamard lacunary sequences.
This was proved in \cite{Le}. (Closely related results appear in \cite{CWW}).
More precisely, it turns out that this characterizes such sequences.
\begin{thm} 
Let $n_0<n_1<\cdots < n_k<\cdots$ be integers.
Let $\Lambda=\{n_k\}$.
The following are equivalent:
\begin{itemize}
\item[(i)] The set $\Lambda=\{n_k\}$ is a finite union of Hadamard lacunary sets.
\item[(ii)] There is a constant $C$ such that for any $x\in \ell_2$  
we have
$$\|\sum\n_k x_k e^{in_k t}\|_{*,\psi_2}\le C (\sum |x_k|^2)^{1/2} .$$
\item[(iii)] There is a constant $C$ such that for any $x\in \ell_2$  
we have
$$\|\sum\n_k x_k e^{in_k t}\|_{*,1}\le C (\sum |x_k|^2)^{1/2} .$$
\end{itemize}
\end{thm}
\begin{proof}  The key fact here is that (i) $\Rightarrow$ (ii).  
It suffices obviously to prove (ii) assuming that the sequence $\{n_k\}$
is itself lacunary.
 This is proved in detail in \cite{Le} to which we refer the reader.\\ 
 (ii) $\Rightarrow$ (iii) is obvious.\\
 Assume (iii). We claim that there is a constant $N$ such that
 for any $n\ge 1$
 $|\Lambda \cap [2^n,2^{n+1}]|\le N$. From this claim, as already mentioned
 it is easy to deduce (i) (see Remark \ref{pal}).
 To prove the claim, fix $n\ge 1$ and let $\varphi_n:\ \N \to \R$ be
 the function 
 defined by the graph in the picture below. More explicitly, $\varphi_n(k)=1$  $\forall  k\in [2^n,2^{n+1}]$,
 $\varphi_n(k)=0$   $\forall k\not\in (0,2^{n+1}+2^{n})$
 and $\varphi_n$ passes affinely
 from  $0$ to $1$ (resp. $1$ to $0$) on the   interval  $[0,2^n]$ (resp.
  $[ 2^{n+1} ,2^{n+1}+2^{n}]$). 
 We then consider the trigonometric polynomial $P_n\in \cl P_+$ defined
  on $\T$ by $P_n(t)=\sum_{k\ge 0} \varphi_n(k) e^{ikt}$, so that
  $\varphi_n$ is the Fourier transform of $P_n$. It is  a well known
  fact that $\|P_n\|_{H^1}\le 2$. To check this observe that $P_n$ is the difference of two Fejer kernels, suitably   translated and scaled, as in the picture below. Explicitly, the classical Fejer kernel,
  which is defined by  $\hat{F_N}(k)=(1-|k|/N)^+$ ($k\in \Z$, $N\ge 1$) satisfies
  $\|F_N\|_1=1$, and 
  we have $$\forall k\in \Z\quad \hat{P_n}(k)= \varphi_n(k)=  
  \frac{3}{2}F_{2^n+2^{n-1}} (k- (2^n+2^{n-1}) )  -\frac{1}{2}F_{2^{n-1}}(k- (2^n+2^{n-1}) ) .$$
  
  \begin{figure}[h]
\centering
\begin{tikzpicture}[scale=1.5]
\draw[->] (0,0)--(0,2);
\draw[->] (0,0)--(8,0);
\filldraw [black] (0,0) circle (1pt);
\filldraw [black] (2.5,0) circle (1pt);
\filldraw [black] (5,0) circle (1pt);
\filldraw [black] (7.5,0) circle (1pt);
\filldraw [black] (0,1) circle (1pt);
\draw (-0.2,1) node {{\tiny $1$}};
\draw (-0.1,-0.2) node {{\tiny $0$}};
\draw (7.5,-0.2) node { {\tiny $2^{n+1} +2^n$}};
\draw (5,-0.2) node { {\tiny $2^{n+1}$}};
\draw (2.5,-0.2) node {{\tiny $2^n$}};
\draw (0,0)--(2.5,1)--(5,1)--(7.5,0);
\draw [dashed] (2.5,0)--(2.5,1);
\draw [dashed] (5,0)--(5,1);
\draw [dashed] (2.5,1)--(3.75, 1.5)--(5,1);
\end{tikzpicture}
\end{figure}

  Let $f=\sum\n_k x_k e^{in_k t}$. 
  By \eqref{feff} we have  
  $$|\sum\n_k  \varphi_n(n_k) x_k/2|= \int   f(t)   \Re(P_n(t) )  dm|
  \le |\int  \Re(f(t))  \Re(P_n(t) )  dm|+ |\int  \Im(f(t))  \Re(P_n(t) )  dm|$$
  $$ \le 
  C_F \|\Re(f) \|_{*,1}\|P_n\|_{H^1}+ C_F \|\Im(f) \|_{*,1}\|P_n\|_{H^1}\le 2C_F \|f \|_{*,1}\|P_n\|_{H^1} \le 4C_FC (\sum |x_k|^2)^{1/2}.$$
  Taking the supremum of the left hand side over all $(x_k)$
  such that  $(\sum |x_k|^2)^{1/2}\le 1$ we obtain
  $$(\sum\n_k \varphi_n(n_k) ^2)^{1/2}  /2 \le 4C_FC .$$
 A fortiori, recalling $\varphi_n(k)=1$  $\forall  k\in [2^n,2^{n+1}]$, this implies
 $$|\Lambda \cap [2^n,2^{n+1}]|\le (8C_FC)^2 .$$
 This proves the claim and concludes the proof.
\end{proof}

 \section{Sidon sets}\label{sid}
The notion of Sidon set, or more generally of ``thin set", has a long history. See
the classical books \cite{HR,LR,GMc}. For a 
 more recent account  see \cite{GH}.
 There are many connections between Sidon sets
 and random Fourier series. See 
\cite{MaPi}
 for more in this direction.
In general Kahane's books \cite{Ka,Ka2} are a wonderful introduction to the use
random functions in harmonic analysis. The many connections with Banach space theory are presented in \cite{LQ}.
\begin{dfn}  
Let $\Lambda=\{\varphi_n\mid n\ge 1\}$ be a bounded sequence
in $L_\infty(T,m)$ ($  (T,m) $ being a probability space).
We say that $\Lambda$ is  Sidon
if there is a constant $C$ such that
for any finitely supported  scalar sequence $(a_n)$ we have
$$\sum |a_n|\le C \|\sum  a_n \varphi_n \|_\infty.$$
\end{dfn}
Note that if $C'=\sup\n_{n\ge 1} \|\varphi_n \|_\infty$ we have
obviously
$$\|\sum  a_n \varphi_n \|_\infty\le C' \sum |a_n|.$$

Let $\Lambda$ be a set of continuous characters on a compact Abelian group $G$.
We may view $\Lambda$ as a subset of $L_\infty(G,m_G)$. 
For instance when $G=\T$ we may identify $\Lambda$
with a subset of $\Z$, and we view 
$$ \{\varphi_n\mid n\ge 1\}=\{ e^{int} \mid n\in \Lambda\}$$

The study of Sidon sets,
or more generally of `thin" sets, was a very active subject in harmonic
analysis in the 1960's and 1970's.
A puzzling problem that played an important role there early on
was the union problem: whether (in the case of sets of characters)
the union of two Sidon sets is a Sidon set. The difficulty is that
if $\Lambda_1$ and $\Lambda_2$ are disjoint sets in $\hat G$
there is a priori no inequality of the form
$$\|\sum\n_{n\in \Lambda_1} a_n \varphi_n \|_\infty \le C  
\|\sum\n_{n\in \Lambda_1\cup \Lambda_2} a_n \varphi_n \|_\infty.$$
 The union problem was eventually solved positively by Sam Drury in 1970
using a very beautiful argument involving
 convolution in measure algebras (see \cite{LR}).

Rider \cite{Ri} refined Drury's trick and
 connected Sidon sets with random Fourier
series. To explain this we need one more definition.
Recall that  $(\vp_n)$ is an i.i.d. sequence
of choices of signs on a probability space $(\Omega,\P)$, i.e.
 $(\vp_n)$ are independent and $\P\{\vp_n=\pm 1\}=1/2$.
\begin{dfn}  
Let $\Lambda=\{\varphi_n\mid n\ge 1\}$ be a bounded sequence
in $L_\infty(T,m)$ ($  (T,m) $ being a probability space).
We say that $\Lambda$ is  randomly Sidon  
if there is a constant $C$ such that
for any finitely supported  scalar sequence $(a_n)$ we have
$$\sum |a_n|\le C\E \|\sum \vp_n a_n \varphi_n \|_\infty.$$
\end{dfn}

\begin{thm}[Rider \cite{Ri}]\label{riri} Let $\Lambda\subset \Z$  or more generally
$\Lambda\subset \hat G$ ($\hat G$   any discrete Abelian group). If  $\Lambda$ is randomly Sidon  then it is Sidon
(and the converse is trivial).
\end{thm}

Rider's proof of this theorem can be interpreted as a refinement
of Drury's, and indeed, Rider's Theorem implies that the union
of two Sidon sets is a Sidon set, because it is easy to check 
that the union of two randomly Sidon sets is randomly Sidon.
Indeed, now if $\Lambda_1$ and $\Lambda_2$ are disjoint sets in $\hat G$
we do have
$$\E\|\sum\n_{n\in \Lambda_1} \vp_n a_n \varphi_n \|_\infty \le   
\E\|\sum\n_{n\in \Lambda_1\cup \Lambda_2} \vp_n a_n \varphi_n \|_\infty.$$

The connection with subgaussian sequences
originates in the following
\begin{thm}[\cite{Ru,Pi}]\label{rud} Let $  \Lambda\subset \Z\setminus \{0\}$  or more generally
$\Lambda\subset \hat G \setminus \{0\}$ ($\hat G$   any discrete Abelian group). Then $\Lambda$ is Sidon if and only if it is subgaussian.
\end{thm}
Rudin proved that Sidon implies subgaussian and asked whether
the converse was true. We proved this in \cite{Pi}, using Gaussian random Fourier
series. Bourgain \cite{Bo}
gave a more direct proof avoiding random Fourier
series.
 In any case, Drury's ideas are still somewhere in the background, and this
is not surprising: indeed, it is obvious (recall \eqref{69}) that  the union
of two subgaussian sequences is a subgaussian sequence.
\begin{proof}[Proof of Theorem \ref{rud}] Assume $\Lambda\subset \hat G$ Sidon. 
Let $M(G)$ be the space of (complex) measures on $G$ equipped
with the total variation norm $\|\mu\|_{M(G)}=|\mu|(G)$.
Recall the identification $M(G)=C(G)^*$.
Let us enumerate $\Lambda=\{ \gamma_n \mid n\in \N\}$.
For any $z=(z_n)\in \T^\N$ and any $f\in {\rm span}(\Lambda)$ we have
$$|\sum z_n \int \ovl{ \gamma_n} f dm_G|\le \sum\n_{\gamma\in \Lambda} |\hat f(\gamma)| \le C \|f\|_{C(G)}. $$
By Hahn-Banach there is a  $\nu_z\in M(G)$
with $\|\nu\|_{M(G)}  \le C$
such that
$\nu_z(f)= \sum z_n \int \ovl{ \gamma_n} f dm_G$ or equivalently 
$\nu_z(\gamma_n)=z_n$ for all $n$.  
Let $\mu_z$ be the symmetric of $\nu_z$ defined by
$\mu_z(f)= \int f(-t)  \nu_z(dt)$.
Then $\hat{\mu_z}(\gamma_n)=z_n$ for all $n$. 
Let $f\in {\rm span}(\Lambda)$, say 
$f=\sum a_n \gamma_n$. Then
$\mu_z\ast f=\sum a_n z_n \gamma_n$ and
$\|\mu_z\ast f\|_p \le C \|{\mu_z}\|_{M(G)}\le C \|f\|_p$.
But we may apply this last inequality also to
$f=\sum a_n \ovl{z_n} \gamma_n$. This gives
us
$$\forall z\in \T^\N\quad \| \sum a_n  \gamma_n\|_p\le C \| \sum z_n a_n  \gamma_n\|_p.$$
Integrating the $p$-th power over $z$ we find
$$  \| \sum a_n  \gamma_n\|_p\le C \left(\int | \sum z_n a_n  \gamma_n(t) |^p dm_G(t) dm_{\T^\N}(z) \right)^{1/p} $$
and by  the translation invariance of $m_{\T^\N}$ this last term is the same as
$C(\int | \sum z_n a_n    |^p  dm_{\T^\N}(z) )^{1/p} $.
Therefore we obtain
$$  \| \sum a_n  \gamma_n\|_p\le C(\int | \sum z_n a_n    |^p  dm_{\T^\N}(z) )^{1/p}. $$
But since we know (see Remark \ref{vp})
that $sg{(z_n)}\le 1$, 
by Lemma \ref{42}
we obtain 
$$  \| \sum a_n  \gamma_n\|_p\le C \beta\sqrt{p}
(\sum |a_n|^2)^{1/2} $$
where $\beta$ is a numerical constant. By Lemma \ref{42}
again, $\Lambda=(\gamma_n)$ is subgaussian.\\
That subgaussian implies Sidon will be fully proved
in a more general framework in the next section (see Remark \ref{ss}).
\end{proof}
\section{Subgaussian bounded orthonormal systems}\label{bole}
Recently 
Bourgain
 and Lewko \cite{BoLe}
 tried to understand what remains true
 for general bounded orthonormal systems 
 of the equivalences described in \S \ref {sid},
 namely the equivalence between Sidon, randomly Sidon and
 subgaussian.
 
 Obviously Sidon $\Rightarrow$ randomly Sidon remains true.
 However, it is easy to see that Sidon $\not\Rightarrow$ subgaussian
  for general  orthonormal systems bounded
 in $L_\infty$. Indeed,
 if $(\varphi_n)$ is Sidon  say on $([0,1], dt)$
 then any system on $([0,2], dt/2)$ that coincides with
 $(\varphi_n)$ on $[0,1]$ is still Sidon, but 
 if its restriction to $[1,2]$ is not subgaussian,
 the resulting system on $([0,2], dt/2)$
 cannot be subgaussian.
For the converse implication,
it turns out to be more delicate to produce a counterexample
but  Bourgain and Lewko \cite{BoLe} managed to do that.
Nevertheless, they proved that subgaussian implies
$\otimes^5$-Sidon in the following sense:
\begin{dfn}
Let $k\ge 1$.
We say that $(\varphi_n)$ is $\otimes^k$-Sidon   with constant $C$
if the system $\{\varphi_n(t_1)\cdots \varphi_n(t_k)\}$ 
(or equivalently $\{\varphi_n^{\otimes k} \}$) is Sidon 
with constant $C$ in $L_\infty(T^k,m^{\otimes k})$.
\end{dfn}
In  \cite{BoLe} they asked whether $5$ could be replaced by $2$,
and in \cite{Pi} we showed that indeed it is so:
\begin{thm}\label{mr1} Any subgaussian system bounded
in $L_\infty(T,m)$ and orthonormal in $L_2(T,m)$ is
$\otimes^2$-Sidon.
\end{thm}
\begin{rem}\label{ss} The preceding result (as well as the previous one
obtaining $\otimes^5$-Sidon) implies the result
stated in Theorem \ref{rud}
that for subsets of $\hat G$ ($\hat G$
discrete Abelian group) subgaussian implies Sidon.
Indeed, if the functions $\varphi_n$
are characters then 
the identity $\varphi_n(t_1\cdots t_k)=\varphi_n(t_1) \cdots\varphi_n(t_k)$
shows that for characters $\otimes^k$-Sidon $\Rightarrow$ Sidon.
 \end{rem}
 
 The key to the proof of Theorem \ref{mr1}
 is the next statement, for which we need to recall the definitions of the
 projective and injective tensor norms, respectively 
 $\|\ \|_{\wedge}$ and $\|\ \|_{\vee}$ on
 the  algebraic tensor product $L_1(m_1)\otimes L_1(m_2)$
 (here $(T_1,m_1),(T_2,m_2)$ are arbitrary measure spaces).
 Let $T\in L_1(m_1)\otimes L_1(m_2)$
 say $T=\sum x_j\otimes y_j$
we set
 $$\|T\|_\wedge=\int |\sum x_j (t_1) y_j (t_2)|dm_1(t_1)dm_2(t_2)$$
 $$\|T\|_\vee=\sup \{|\sum \langle x_j ,\psi_1\rangle   \langle y_j ,\psi_2\rangle |\mid \|\psi_1\|_\infty \le 1,
 \|\psi_2\|_\infty\}.$$ 
 Note that the completion of $L_1(m_1)\otimes L_1(m_2)$ with respect to
  $\|\ \|_{\wedge}$ can be identified isometrically to 
  $L_1(m_1\times m_2)$.

\begin{thm}\label{nt1} Let $(T ,m ) $
 be a probability space.
 Let $(g_n)$ be an i.i.d. sequence of normalized $\R$-Gaussian  random
 variables.
  For any $0<\d<1$ there is 
  $  w(\d)>0$ for which the following property holds.
  Let $\{\varphi _n\mid 1\le n\le N\}\subset   L_1(m)$
  be any system that is 
     $C$-dominated by $\{g_n\mid 1\le n\le N\}$.  Then,
     for any $(z_n)\in \C^N$ with $|z_n|\le 1$, there is a decomposition in $L_1(m)\otimes L_1(m)$
  of the form
    \begin{equation}\label{n0}\sum\nolimits_1^N z_n \varphi_n \otimes \varphi_n=t+r\end{equation}
 satisfying
  \begin{equation}\label{n1}\|t\|_{\wedge}\le Cw(\d)\quad \text{   and   }\quad  \|r\|_{\vee}\le C\d.\end{equation}
       \end{thm}
 \begin{proof} It clearly suffices to prove this
in the case $\varphi _n=g_n$ and $C=1$ (indeed, the classical properties of tensor products
allow us to pass from $g_n$ to $\varphi _n$). Moreover,
treating separately $\sum\n_1^N \Re(z_n) \varphi_n \otimes \varphi_n$
and $\sum\n_1^N \Im(z_n) \varphi_n \otimes \varphi_n$,
we may reduce to the case when the $z_n$'s are in $[-1,1]$.
But then, by Lemma \ref{cm}    there is an operator
$\Theta_z:\ L_1(\P) \to L_1(\P)$ with norm $1$ such that 
$T_z(g_n)=z_n g_n$. Using this, we can reduce to the case
when $z_n=1$ for all $n$.
We will show that   Theorem \ref{nt1} can be easily derived
from the following\\
\noindent{\bf Claim:} for any $0<\d<1$
there is $\Phi\in L_1(\P\times \P)$ with $\|\Phi\|_{L_1(\P\times \P)}=1$
such that 
$$\Phi =1\otimes 1 + \d \sum\nolimits_1^N   g_n \otimes g_n
+ R$$
where $R$ viewed as an operator
on $L_2(\P)$ has norm $\le \d^2$.
 This claim is immediate from the discussion in \S \ref{meh}. We just take
for $\Phi$  the Mehler kernel and note that
 $P_1$ can be identified with $\sum\nolimits_1^N   g_n \otimes g_n$ and we have
 $\|\sum\n_{d\ge 2} \d^d
P_d:\ L_2(\P)\to L_2(\P)\|\le \d^2$.\\
From the claim we deduce
$$\sum\nolimits_1^N   g_n \otimes g_n= t'+r'$$
with $t'=(\Phi-1\otimes 1)/\d$ and $r'=-R/\d$. Then we have\\
\centerline{
$\|t'\|_{L_1(\P\times \P)}\le 2/\d$ and
$\|r'\|_{\vee}\le \|r:\ L_2(\P)\to L_2(\P)\|\le \d$.}
The only problem is that $t$ (and hence also $r$) are in the space
$L_1(\P\times \P)$ and we want them to be in 
$L_1(\P)\otimes L_1(\P)$. In other words
we want the associated operators to be of finite rank.
This can be fixed like this: it is a well known property of $L_1$-spaces that
for any $\vp>0$ and any finite dimensional subspace $E\subset L_1$
there is a finite rank operator $v:\ L_1\to L_1$ with $\|v\|<1+\vp$
that is the identity on $E$.
We apply this to $E={\rm span}[g_n\mid 1\le n\le N]$ with
(say) $\vp=1$, and then we set
$t=(v\otimes Id)(t')$ and
$r=(v\otimes Id)(r')$. This gives us 
finite rank tensors satisfying the desired conclusion
with $\|t\|_\wedge\le 4/\d$ and  $\|r\|_\vee \le 2\d$. Since we may trivially
replace $\d$ by $\d/2$, the proof is complete.
\end{proof}
\begin{proof}[Proof of Theorem \ref{mr1}] 
Let $C'=\sup\n_n\| \varphi_n\|_\infty$.
Note that $(\varphi_n)$ is subgaussian iff $(\ovl{\varphi_n})$
  also is, with the same constant.
By  Proposition \ref{p8}, any subgaussian system
is $C$-dominated by $(g_n)$ for some $C$.
   Let $z_n\in \T$ be such that $|a_n|=\vp_n a_n$.
      Let
      $\sum z_n \ovl{\varphi_n} \otimes \ovl{\varphi_n}=t+r$ as in \eqref{n1}.
      Let $f(t_1 ,t_2)= \sum a_n \varphi_n (t_1 )\varphi_n (t_2 )$.
      We have
      $$\langle t+r,f\rangle=\int (\sum  z_n \ovl{\varphi_n} \otimes \ovl{\varphi_n}) f=\sum z_n a_n=
      \sum |a_n| .$$
      Therefore
      $$\sum |a_n| \le |\langle t,f\rangle| +  |\langle r,f\rangle|
      \le Cw(\d) \|f\|_\infty + \sum |a_n| |\langle r,\ovl{\varphi_n}\otimes \ovl{\varphi_n}\rangle|
      \le  Cw(\d) \|f\|_\infty +C\d {C'}^2  \sum |a_n| .$$
Choosing $\d$ such that $\d C {C'}^2  =1/2$
we   conclude
that 
$ (\varphi _n\otimes \varphi _n)$ is Sidon with constant $\le 2 C w(\d) $.
       \end{proof}

\begin{rem} It is proved in \cite{Pi3} that, in the situation of Theorem
\ref{mr1} $(\varphi_n)$ is randomly Sidon iff
it is $\otimes^k$-Sidon for some (or equivalently for all) $k\ge 4$.
This extends Rider's Theorem \ref{riri}  to bounded orthonormal systems.
Here, the cases $k=2$  and $k=3$ remain open.
  \end{rem}
  \begin{rem} 
  Let $(\varphi_n)$ be uniformly bounded and orthonormal.
  The same interpolation argument alluded to
  in Remark \ref{ia} shows that  if $(\varphi_n)$ is subgaussian
  (or if it merely satisfies the analogue of (iv) in Theorem \ref{erd})
  then for any $1<p<2$ there is a constant $C_p$
  such that for any $f=\sum a_n \varphi_n$ in its linear span
  we have
   \begin{equation}\label{f74}  \|f\|_{\psi_{p'}} \le C_p (\sum  |a_n|^p)^{1/p}.\end{equation}
  Actually, one can even prove
  $  \|f\|_{\psi_{p'}} \le C_p \|(a_n)\|_{p,\infty},$
  where $\|(a_n)\|_{p,\infty}=\sup\n_{n\ge 1} n^{1/p} a_n^*$
  (here $a_1^*\ge a_2^*\ge\cdots\ge a_n^*\ge\cdots$ is the non-increasing rearrangement
  of $(|a_n|)$.
  
  \medskip
  
\noindent{\bf Problem:} Does \eqref{f74}
imply that $(\varphi_n)$ is 
$\otimes^k$-Sidon for some   $k> 1$ ?
  \end{rem}

 \end{document}